\documentclass[12pt,oneside]{article}
\newcommand{\Path}{}
\usepackage{ \Path Paper_style_2}
\usepackage{ \Path Keywords_WB}
\usepackage{ \Path Environments}
\usepackage{ \Path Category}
\newcommand{\Diagrams}{}
\usepackage{tikz-cd}
\usepackage{ \Path tikzcd2}

\renewcommand{\Chain}{\text{Chain}}
\renewcommand{\Cone}{\Psi}
\begin{document}
	\title{Homology and homotopy in general categories}
	\author{Suddhasattwa Das\footnotemark[1]}
	\footnotetext[1]{Department of Mathematics and Statistics, Texas Tech University, Texas, USA}
	\date{\today}
	\maketitle
	
	\begin{abstract} 
		One of the prime motivation for topology was Homotopy theory, which captures the general idea of a continuous transformation between two entities, which may be spaces or maps. In later decades, an algebraic formulation of topology was discovered with the development of Homology theory. Some of the deepest results in topology are about the connections between Homotopy and Homology. These results are proved using intricate constructions. This article re-proves these connections via an axiomatic approach that provides a common ground for homotopy and homology in arbitrary categories. One of the main contributions is a re-interpretation of convexity as an extrinsic rather than intrinsic property. All the axioms and results are applicable for the familiar context of topological spaces. At the same time it provides a complete framework for an algebraic characterization of objects in a general category, which also preserves a notion of Homotopy.
	\end{abstract}
	%-_-_-_-_-_-_-_-_-_-_-_-_-_-_-_-_-_-_-_-_-_-_-_-_-_-_-_-_-_-_-_-_-_-_-_-_-_-_-_-_-_-_-_-_-_-_-_-_-_-_-_-_-_-_-_-_-_-_-_-_-_-_-_-_-_-_-_-_-_-_-_-_-_-_-_-
	\section{Introduction} \label{sec:intro}
	
	\begin{figure}[!ht]
		\centering
		\begin{tikzpicture}[scale=1.0, transform shape, framed, background rectangle/.style={double, ultra thick, draw=gray, rounded corners}]
			\input{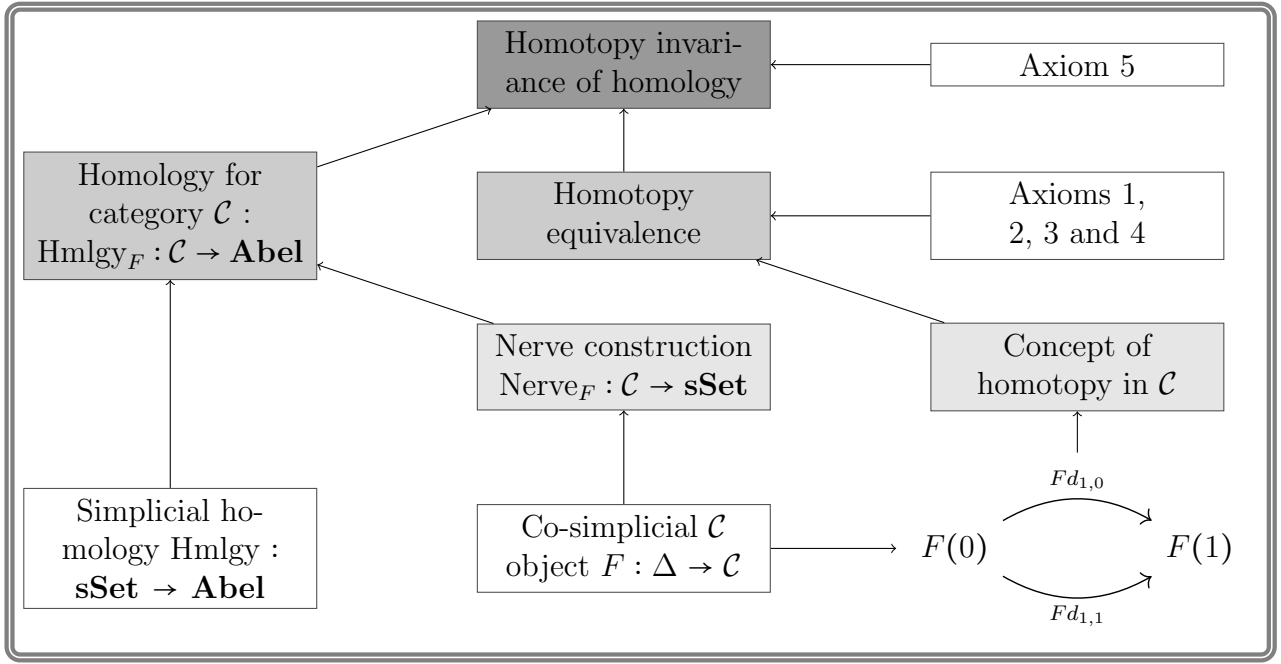}
		\end{tikzpicture}
		\caption{Outline of the theory. This article presents a simple axiomatic approach to both homotopy and homology in a general category $\calC$. A collection of assumptions on $\calC$ not only help generalize these concepts but also guarantees that the homology preserves homotopy invariance. The chart above presents various assumptions and concepts, along with their logical dependence. The independent notions are in white boxes. This includes Simplicial homology, a means of converting the combinatorial structure of a simplicial set into an Abelian group. Homology is the confluence of simplicial homology and the nerve construction, as shown in detail in \eqref{eqn:CatHmlgy}. The nerve construction \eqref{eqn:def:Nerve} is an outcome of an arbitrary functor $F$ as in \eqref{eqn:functorF}. A bare minimum notion of Homotopy can be built from the image of the first two objects of the simplex category $\Delta$ under functor $F$. To attain the more developed notion of Homotopy equivalence one needs further axioms as indicated. Figure \ref{fig:hmtpy_outline} presents a detailed outline of this property. Also see Figure \ref{fig:hmlgy_hmtpy} for a detailed outline of the derivation of homotopy invariance from more fundamental properties.}
		\label{fig:paper_outline}
	\end{figure} 

The two pillars of modern topology are Homotopy theory and Homology theory. The former studies the continuous deformation of maps and spaces, and the various equivalences they create. On the other hand, homology provides a means of embedding the arrangement of topological spaces and continuous maps into a system of Abelian groups and group homomorphisms between them. The axiomatic approaches for homotopy \cite[e.g.]{Quillen2006hmtpy, Riehl2014homotopy} and Homology \cite[]{Segal1968classify, Kan1957css, Kan1958fnctrs, Andre1970hmlgy} have mostly been independent. Subsequently, a number of axiomatic frameworks have been suggested to unify these two theories. The theory of model categories \cite[e.g.]{Quillen2006hmtpy, DwyerSpalinski1995homotopy} extends some notions from topology by formulating three abstract notions of morphisms - fibrations, cofibrations and weak equivalences. Another approach is that of $\infty$-categories \cite[e.g.]{Cisinski2019higher}, which is a reformulation of category theory itself so that notions of homotopy gets ingrained into the notions. Some other innovations in this field can be found in \cite[e.g.]{GoerssJardine2009, DwyerKan1984singular}. 

Our goal is also to present a general axiomatic framework to establish the notions of homotopy and Homology, and achieve the homology invariance with respect to homotopy equivalences. The axioms to be presented are simpler, verifiable, and have intuitive interpretations. They do not require the heavy structural assumptions of model or $\infty$-categories. Rather they share some similarities with classical \textit{generalized homology theory} \cite{EilenbergSteenrod2015topo} which interprets homology as a sequence of functors $H_n$ satisfying a set of five axioms. Some of the most important generalized homology theories arise from the study of \textit{stable homotopy theory} \cite{adams1974stable} and are usually represented by a construct call \textit{spectra}. Homology is thus interpreted fundamentally as a form of stable homotopy. Our axiomatization involves only two of the Eilenberg Steenrod axioms - homotopy and dimension axioms. Moreover, they occur not as axioms but as consequences of even more simpler assumptions on the category.

The need for Homotopy and Homology theories for general categories is still relevant due to the increasing categorification of various branches of Mathematics. Categories and category theory provide a panoramic view of many of the various branches of mathematics, such as Dynamical systems theory \cite[]{MossPerrone2022ergdc, Das2024slice, Das2023CatEntropy}, Measure theory \cite[]{Leinster2020integration, Panangaden1999Markov}, Algorithms \cite[]{Yanofsky2011algo, Yanofsky2022theoretical} or Game theory \cite[]{GhaniHEdges2006game, GhaniEtAl2018iter}. In fact, this discipline arose from a systematic study of topological objects and continuous map, and its connections with Algebra. The concepts of Homology and Homotopy were respectively established as functor from the category of topological spaces into the category of Abelian groups, and from the category of pointed topological spaces into the category of general Groups. The concept of a functor provides a concise description of a mapping that also preserves the numerous relations between two arrangements. Readers can find the basic definitions of category theory in standard sources such as \cite{Riehl2017context, Maclane2013}. Our axiomatization starts with the combinatorial notion of \emph{simplices}.

\paragraph{Simplex category} The simplex category $\Delta$ \cite{mcinnes_umap_2018, Riehl2011sset} is a foundational structure in algebraic topology and category theory, primarily serving as a combinatorial blueprint for building and studying simplicial objects. It has as objects the non-negative integers $\num_0 :- \braces{0, 1,2,\ldots }$. Each such integer $n$ is meant to represent the ordered sets $[n] := \{0, \ldots, n\}$. The morphisms are given by
%m
\[ \Hom_{\Delta} (m;n) := \braces{ \mbox{ Order preserving maps } \phi:[m] \to [n] } . \]
The simplex category $\Delta$ is the most concise way of encoding combinatorial structures in other categories \cite{Riehl2011sset, Steiner2007omega, Grodal2002higher}. The goal of this article is to present how a simple functor 
\begin{equation} \label{eqn:functorF}
F : \Delta \to \calC
\end{equation}
mapping the simplex category $\Delta$ into an arbitrary category $\calC$, leads to notions of Homotopy and Homology for the category $\calC$. The functor $F$ creates an image of $\Delta$ within $F$. The objects $\SetDef{F(n)}{n\in\num_0}$ will be called \emph{cells}. They shall serve as the basic building blocks of homology and homotopy. See Figure \ref{fig:paper_outline} for a concise outline of this article.

\paragraph{Topological simplex} One of the main motivating examples is the case when $\calC = \Topo$, the category of topological spaces and continuous maps. In that case $F$ is often taken to be \emph{standard topological simplex}
\begin{equation} \label{eqn:def:StndrdTopoSmplx}
F_{stndrd, Topo} : \Delta \to \Topo, \quad F_{stndrd, Topo} (n) = \mbox{ convex span of } \braces{ e^{(n+1)}_0, \ldots, e^{(n+1)}_n } ,
\end{equation}
where $e^{(n+1)}_0, \ldots, e^{(n+1)}_n$ are $n+1$ independent eigenvectors in $\real^{n+1}$. The resulting homology is the singular homology for topological spaces. Note that the $n$-th cell of this functor is isomorphic to the $n$-dimensional closed disk. Our approach is not tied to to the specific example of \eqref{eqn:def:StndrdTopoSmplx}. We shall identify and isolate structural properties of the general functor \eqref{eqn:functorF} and recreate the homology and homotopy. The specific functor in \eqref{eqn:def:StndrdTopoSmplx} will serve as an useful example to verify our axioms with.

Our first two axioms will be on the basic elements of $F$ :

\begin{Axiom} \label{A:C}
The category $\calC$ has a terminal object $1_{\calC}$, and has finite products.
\end{Axiom}

and

\begin{Axiom} \label{A:1_0_cell}
The functor $F$ maps $[0]$ into $1_\calC$ from Axiom \ref{A:C}.
\end{Axiom}

Axiom \ref{A:1_0_cell} simply states that the 0-th cell is the terminal object of $\calC$. Note that for the example of \eqref{eqn:def:StndrdTopoSmplx}, $1_{\calC}$ is the one-point topological space, which is also the $0$-th cell. The Cartesian products of spaces in $\Topo$ also coincides with categorical products.

\paragraph{Generating morphisms} In general, it is not easy to make an explicit construction of a functor of the form \eqref{eqn:functorF}, as the morphisms incoming at $n$ or outgoing at $n$ increase exponentially with $n$. The task becomes easier if one concentrates only on a special collection of morphisms known as \emph{face maps} :
\begin{equation} \label{eqn:def:face}
0\leq i\leq n \;:\; \Shobuj{d_{n,i}} : [n-1] \to [n], \, \quad j \mapsto \begin{cases}
	j & \mbox{ if } j<i	\\
	j+1 & \mbox{ if } j\geq i	
\end{cases}.
\end{equation}
and \emph{degeneracy} maps :
\begin{equation} \label{eqn:def:degeneracy}
0\leq i\leq n \;:\;  \Shobuj{s_{n,i}} : [n+1] \to [n], \quad j \mapsto \begin{cases}
	j & \mbox{ if } j \leq i	\\
	j-1 & \mbox{ if } j > i	
\end{cases}.
\end{equation}
To put more concisely, the $i$-th face map $d_{n,i}$ is the unique map which skips the element $i$ in $[n+1]$, and the $i$-th degeneracy map $s_{n,i}$ is the unique surjective map such that the pre-image of element $i$ has two elements. The face and degeneracy maps generate every morphism $\phi:[m]\to [n]$ in $\Delta$. The obey the following set of identities :
\begin{equation} \label{eqn:smplc_id}
\begin{split}
	d_{n+1,i} d_{n,j} = d_{n+1,j+1} d_{n,i} & \mbox{ if } i<j \\
	s_{n-1,j} d_{n,i} = d_{n-1,i} s_{n-2,j-1} & \mbox{ if } i<j \\
	s_{n-1,j} d_{n,j} = \Id_{[n-1]} & \\
	s_{n-1,j} d_{n,j+1} = \Id_{[n-1]} \\
	s_{n-1,j} d_{n,i} = d_{n-1,i-1} s_{n-2,j} & \mbox{ if } i>j+1 \\
	s_{n-1,j} s_{n,i} = s_{n-1,i} s_{n,j+1} & \mbox{ if } i\leq j 
\end{split} 
\end{equation}
which are called the \emph{simplicial identities}. Equations \eqref{eqn:def:face}, \eqref{eqn:def:degeneracy} and \eqref{eqn:smplc_id} provide a sufficient set of morphisms and composition relations that generate all the morphisms of $\Delta$. This means that given a collection of morphisms \eqref{eqn:def:face} and \eqref{eqn:def:degeneracy} which satisfy \eqref{eqn:smplc_id}, any morphism in $\Delta$ may be expressed as a composition of different $d_{n,i}$ and $s_{n',i'}$-s. Thus when trying to construct a functor $F$ as in \eqref{eqn:functorF}, one only needs to specify the objects $\SetDef{ F(n) }{ n\in \num_0 }$, the action of $F$ on the face and degeneracy morphisms, and ensure that $F$ preserves the simplicial identities of \eqref{eqn:smplc_id}. 

The next axiom we need is about the $1$-th cell :

\begin{Axiom} \label{A:swap}
There is an isomorphism $\text{swap} : F(1) \to F(1)$ such that the following commutation holds with boundary maps :
\begin{equation} \label{eqn:def:hmtpy:1}
	\begin{tikzcd}
		& & F(1) \arrow[d, "\text{swap}", bend left=49] \\
		1_{\calC} \arrow[rr, "F\paran{ d_{1,1} }"'] \arrow[rru, "F\paran{ d_{1,0} }", bend left] & & F(1) \arrow[u, "\text{swap}^{-1}", bend left=49]
	\end{tikzcd}
\end{equation}
\end{Axiom}

We shall see later how the $0$-th cell $F(0)$ and $1$-th cell $F(1)$ alone create a notion of homotopy. Based on nomenclature used in the axiomatic theory of sets \cite{Lawvere1964set, Leinster2014Set, Goldblatt2014topoi, Leinster2010topos}, morphisms originating from the terminal element $F(0)$ are called \emph{elements} of the target object. Thus according to \eqref{eqn:def:face}, the face maps $d_{1,0}$ and $d_{1,1}$ are two elements of $F(1)$. The 1-cell $F(1)$ acts as a bridge between these two distinguished elements. The terminal elements this play the role of entry and exit points. Axiom \ref{A:swap} establishes a symmetry of these elements. Note that for the example of \eqref{eqn:def:StndrdTopoSmplx}, the morphism $\text{swap}$ is simply the map that reverses the orientation of the unit interval, which is also the $1$-cell. Such a reversal also interchanges the endpoints, which are the images of the $0$-cell under the face maps.	

To prove transitivity of homotopy, we shall need the following assumption

\begin{Axiom} \label{A:F1_join}
The diagram 
\begin{equation} \label{eqn:dip3d}
	\begin{tikzcd} [column sep = large]
		F(1) & F(0) \arrow[l, "F\paran{d_{1,0}}"'] \arrow[r, "F\paran{d_{1,1}}"] & F(1)
	\end{tikzcd}
\end{equation}
has a colimit, which is $F(1)$ itself.
\end{Axiom} 

For the example of \eqref{eqn:def:StndrdTopoSmplx}, the colimit in \eqref{eqn:dip3d} takes the form  
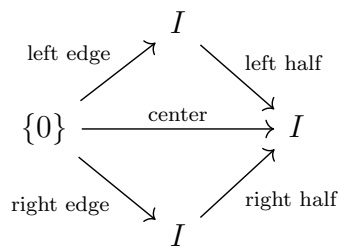
\[\begin{tikzcd}
& I \arrow[dr, "\text{left half}"] \\
\{0\}  \arrow[ur, "\text{left edge}"] \arrow[dr, "\text{right edge}"'] \arrow[rr, "\text{center}"] && I \\
& I \arrow[ur, "\text{right half}"'] 
\end{tikzcd}\]
In fact this nature of the interval $I$ as a gluing of two copies of itself leads to a categorical notion of \emph{self-similarity} \cite{Leinster2011selfsim, Freyd2008algebraic}. It provides the structural basis for composition of homotopies.

\paragraph{Nerve construction} The purpose of the nerve construction is to establish a fundamental link between the purely algebraic structure of the category $\calC$ and the combinatorial/geometric structure of a simplicial set. The general functor $F$ induces a special functor called the \emph{nerve} of the category $\calC$ \cite{Riehl2011sset} as shown below :
\begin{equation} \label{eqn:def:Nerve}
\begin{split}
	& \Nerve = \Nerve_F : \calC \to \sSet, \\ 
	& \forall c\in ob(\calC) \;:\; \Nerve(c) := \Hom \paran{ F\cdot; c } : \Delta^{op} \to \SetCat .
\end{split}
\end{equation}
The functor $\Nerve_F$ assigns a simplicial set to each object of $\calC$. Thus it assigns a purely combinatorial identity to each object in $\calC$. This functor will be the key ingredient to creating a notion of Homology, as we present later in \eqref{eqn:CatHmlgy}. For the example of \eqref{eqn:def:StndrdTopoSmplx}, $\Nerve_F(X)$ is the collection of all possible continuous mappings of the $n$-dimensional simplex into the topological object $X$. Thus $\Nerve_F(X)$ becomes a ledger for the topological ``content" of $X$, and the entries of the ledger are all possible embeddings of $n$-dimensional disks. This feature of the nerve functor has found use in several categorical investigations, such as K-theory \cite[e.g.]{Schwanzl1994basic}, abstractions of homotopy theory \cite[e.g.]{Rezk2001model, bergner2009complt}, and dendroidal sets \cite[e.g.]{Cisinski2013dendro, Cisinski2011dendro}.

In the example of \eqref{eqn:def:StndrdTopoSmplx}, it is easy to visualize the simplices being built inductively. $F(n+1)$ can be built by erecting a convex tent above $F(n)$. With this visual  interpretation, each simplex can be interpreted to be the base / bottom face of the next higher simplex. The $i$-th face of the $n$-simplex becomes the $i+1$-th face of the $n+1$-simplex. We make this nesting relation more precise for the general functor \eqref{eqn:functorF}.

Our last and fifth axiom relies on the notion of natural transformations between functors. 
%\paragraph{Base maps} 
Let $\tilde{\Delta}$ be the subcategory of $\Delta$ generated only by the face maps $\SetDef{ d_{n,i} }{ 0 \leq i \leq n }$. Let $\tilde{\Delta} \xrightarrow{i} \Delta$ denote the obvious inclusion of categories. The simplified simplex $\tilde{\Delta}$ has an in-built $\text{Shift}$ functor : 
\begin{equation} \label{eqn:def:Shift}
\begin{tikzcd} \tilde{\Delta} \arrow[rr, "\text{Shift}"] && \tilde{\Delta}  \end{tikzcd}, \quad 
\begin{tikzcd} \Holud{n} \arrow[d, "d_{n,i}"'] \\ \akashi{n+1} \end{tikzcd} 
\begin{tikzcd} {} \arrow[r, mapsto] & {} \end{tikzcd}
\begin{tikzcd} \Holud{n+1} \arrow[d, "d_{n+1,i+1}"] \\ \akashi{n+2}	\end{tikzcd}
\end{equation}
A key observation is the existence of a natural transformation $\Base$ as shown below :
\begin{equation} \label{eqn:def:Base}
\begin{tikzcd}
	& \Delta^{op} \arrow{rrr}{\Nerve_{F}(X)} &&& \SetCat \\
	\tilde{\Delta}^{op} \arrow[dashed]{urrrr}[name=n1]{} \arrow[d, "\text{Shift}"'] \arrow[dashed]{drrrr}[name=n2]{} \arrow[d, "\text{Shift}"'] \arrow[ur, "\iota"] \\ 
	\tilde{\Delta}^{op} \arrow[r, "\iota"] & \Delta^{op} \arrow{rrr}[swap]{\Nerve_{F}(X)} &&& \SetCat
	\arrow[shorten <=1pt, shorten >=1pt, Rightarrow, to path={(n2) to[out=90,in=-90] node[xshift=20pt]{$\Base$} (n1)} ]{  }
\end{tikzcd} , \quad 
\begin{tikzcd}
	\Hom_{\calC} \paran{ F(n+1); X } \arrow{d}{ \Base_n } [swap]{ \circ F \paran{ d_{n+1,0} } } \\ \Hom_{\calC} \paran{ F(n); X }
\end{tikzcd} , \quad 
\forall n\in \num_0 ,
\end{equation}
whose connecting morphisms are provided by right-composition with the zeroeth face maps. The natural transformation of \eqref{eqn:def:Base} always exists uniquely, and will be verified in Section \ref{sec:cnvxty}. We next formulate a purely categorical notion of convexity. 

\paragraph{Convexity} Any object $X$ will be called \emph{convex} if it has the a natural transformation $\Cone^X$ which plays the role of a right inverse of $\Base$ : 
\begin{equation} \label{eqn:def:ConeNat}
\begin{tikzcd}
	    & && & \Delta^{op} \arrow{drrr}{\Nerve_{F}(X)} \\
\Delta^{op} \arrow{ddr}[swap]{ \Nerve_{F}(X) } & && & {} &&& \SetCat \\
& {} && \tilde{\Delta}^{op} \arrow[dashed, bend right=10, Akashi]{dll}[name = n3]{} \arrow{ulll}[swap]{\iota} \arrow[dashed, Akashi]{urrrr}[name=n1]{} \arrow[dd, pos=0.3, "\text{Shift}"'] \arrow[dashed, Akashi]{drrrr}[name=n2]{} \arrow[uur, pos=0.8, "\iota"] \\ 
& \SetCat && {} & {} &&& \SetCat \\
& && \tilde{\Delta}^{op} \arrow[r, "\iota"] & \Delta^{op} \arrow{urrr}[swap]{\Nerve_{F}(X)}
\arrow[shorten <=1pt, shorten >=1pt, Rightarrow, Shobuj, to path={(n2) to[out=90,in=-45] node[xshift=20pt]{$\Base$} (n1)} ]{  }
\arrow[shorten <=1pt, shorten >=1pt, Rightarrow, Shobuj, to path={(n3) to[out=-90,in=-90] node[yshift=15pt]{$\Cone^X$} (n2)} ]{  }
\arrow[shorten <=1pt, shorten >=1pt, Rightarrow, Shobuj, to path={(n1) to[out=135,in=90] node[yshift=15pt]{$\Id$} (n3)} ]{  }
\end{tikzcd}
\end{equation}
The thick arrows in \eqref{eqn:def:ConeNat} represent natural transformations. This arrangement is one of the major contributions of this article. Convexity is usually interpreted as an intrinsic property of a space, and resulting form a closure property with respect to convex linear sums. Such a linear structure may not be available for a general object or category. Equation \eqref{eqn:def:ConeNat} redefines convexity as a relational property. These relations concisely contained in the property of the transformation $\Cone$ being \emph{natural}. For the topological example of \eqref{eqn:def:StndrdTopoSmplx}, $\Cone$ becomes the \emph{Cone-construction}. There its connecting morphisms are maps $\Cone^X_n : \Hom \paran{ F(n) ; X } \to \Hom \paran{ F(n+1) ; X }$. Recall that the standard topological $n$-simplex can be described by coordinates 
\[ F_{stndrd, Topo} (n) := \SetDef{ \paran{ t_0, \ldots, t_n } }{ t_j \geq 0, \, \sum_{j=0}^{n} t_j = 1 } , \quad \forall n\in \num_0.\]
Let $a_n$ be an arbitrary point in the $n$-th simplex. Using these coordinates, the action of the (topological) cone transformations may be described as
\begin{equation} \label{eqn:Cone_Topo}
\paran{ \Cone^X_n \sigma } \paran{ t_0, \ldots, t_n, t_{n+1} } :=  
\begin{cases}
	t_{0} a_n + \paran{1-t_{0}} \sigma \paran{ \frac{t_1}{1-t_{0}} , \ldots, \frac{t_{n+1}}{1-t_{0}} } & \mbox{ if } t_{0} < 1 \\
	a_n & \mbox{ if } t_{0}  =1 
\end{cases}, 
\quad \forall n\in \num_.
\end{equation}
In $\Topo$, the cells $F(n)$ are obviously convex from their very definition in \eqref{eqn:def:StndrdTopoSmplx} as a convex hull. We formally state this as our final axiom :

\begin{Axiom} \label{A:convex}
All the cells $\SetDef{F(n)}{n\in\num_0}$ are convex.
\end{Axiom}

This completes the formulation of our five basic axioms.

\paragraph{Main result} The axioms we have stated are based the general functor \eqref{eqn:functorF}. Axiom \ref{A:C} is entirely about about the category $\calC$, while the next three axioms \ref{A:1_0_cell}, \ref{A:swap} and \ref{A:F1_join} are entirely about the first two cells of the functor $F$. Axiom \ref{A:convex} is the only assumption made on the functor $F$ as a whole. Our main result is : 

\begin{theorem} [Main result] \label{thm:1}
Consider the functor $F$ as in \eqref{eqn:functorF} mapping the simplex category $\Delta$ into an arbitrary category $\calC$. 
	\begin{enumerate} [(i)]
		\item $F$ creates a homology functor for the category $\calC$, as defined in \eqref{eqn:CatHmlgy}.
		\item Under Axioms \ref{A:C}, \ref{A:1_0_cell}, there is a notion of homotopy between morphisms of $\calC$, as defined in \eqref{eqn:def:hmtpy:2}.
		\item In addition to Axioms \ref{A:C}, \ref{A:1_0_cell}, Axiom \ref{A:swap} guarantees that homotopy is symmetric.
		\item In addition to Axioms \ref{A:C}, \ref{A:1_0_cell}, Axiom \ref{A:F1_join} guarantees that homotopy is transitive.
		\item In addition to Axioms \ref{A:C}, \ref{A:1_0_cell}, Axiom \ref{A:convex} guarantees that homology is homotopy invariant.
\end{enumerate}
\end{theorem}

\paragraph{Remark} Theorem \ref{thm:1} provides a separation of various concepts, as outlined in Figure \ref{fig:paper_outline}. The separate claims list the precise Axioms on which the notions of Homotopy and Homology depend. The construction of Homology requires no assumption other than the functor $F$ \eqref{eqn:functorF}. Homotopy requires Axioms \ref{A:C}, \ref{A:1_0_cell}, \ref{A:swap} and \ref{A:F1_join}. Figure \ref{fig:hmtpy_outline} and Section \ref{sec:hmtpy} further delineates the precise role played by each axiom in constructing the various components of homotopy. The fifth axiom \ref{A:convex} of convexity is required to establish the inter-dependence of homology and homotopy. Also see Figure \ref{fig:paper_outline} for a summary of this main result. 

\paragraph{Remark} There are constructs of homology that are purely combinatorial in nature, such as the homology of simplicial complexes. Those lie outside the scope of Theorem \ref{thm:1}.

\paragraph{Remark} The major innovation in Theorem \ref{thm:1} is the Axiom \ref{A:convex} on convexity. Convexity is usually interpreted via linear convex sums. Axiom \ref{A:convex} extracts the structural essence of convex, using the language of natural transformations. See Theorem \ref{thm:cnvxty} where convexity, a property borne by $\calC$, implies acyclicity (zero-homology), a property borne by the homology functor.

\paragraph{Outline}  As mentioned before Theorem \ref{thm:1} is a statement of some classical results under a new re-axiomatization. Claim~(i) is a standard result and is proved in Section \ref{sec:hmlgy}. Claims (ii) -- (iv) are proved in Section \ref{sec:hmtpy}. There, a generalized notion of homotopy is proposed,  properties such as reflexivity, symmetry, transitivity and homotopy equivalence are derived.  The crux of the Theorem's statement is claim (v), which is proved over the course of Sections \ref{sec:hmlgy_htpy:2} and \ref{sec:cnvxty}. In the process, we identify several intermediary properties which guarantee the invariance of homology with homotopy. The axiomatic approach that we present in this article open several new directions of investigation. These are discussed briefly in Section \ref{sec:conclus}. Finally, we provide some lemmas and proofs in Section \ref{sec:proofs}.
%-_-_-_-_-_-_-_-_-_-_-_-_-_-_-_-_-_-_-_-_-_-_-_-_-_-_-_-_-_-_-_-_-_-_-_-_-_-_-_-_-_-_-_-_-_-_-_-_-_-_-_-_-_-_-_-_-_-_-_-_-_-_-_-_-_-_-_-_-_-_-_-_-_-_-_-

%-_-_-_-_-_-_-_-_-_-_-_-_-_-_-_-_-_-_-_-_-_-_-_-_-_-_-_-_-_-_-_-_-_-_-_-_-_-_-_-_-_-_-_-_-_-_-_-_-_-_-_-_-_-_-_-_-_-_-_-_-_-_-_-_-_-_-_-_-_-_-_-_-_-_-_-
\section{Homology} \label{sec:hmlgy} 

We now review the creation of Homology from a co-simplicial object such as $F$ in \eqref{eqn:functorF}. Homology is a functor, labeled as $\Homology_{\calG, F}$ in the diagram below.
\begin{equation} \label{eqn:CatHmlgy}
\begin{tikzcd}[scale cd = 1.5]
	\calC \arrow[dashed, Shobuj]{ddrr}[swap]{ \Homology_{\calG, F} } \arrow[Akashi]{rrrr}{ \Nerve_{F} } && && \sSet \arrow[dashed, Holud]{ddrr}[swap]{ \Chain_{\calG} } \arrow[dashed, Itranga]{ddll}[swap]{ \text{SimpHom}_{\calG} } \arrow{rr}{\otimes \calG} && \sGroup \arrow{dd}{ \text{alt} } \\
	\\
	&& \AbelCat^{\num} && && \ChainCmplx \arrow{llll}{ \Homology_{\calA} }
\end{tikzcd} 
\end{equation} 
All the dashed arrows in this diagram are defined via composition. $\calG$ denotes an Abelian group, most commonly taken to be $\integer$. Homology can thus be simply stated as the composition of two ingredients -- the simplicial homology functor (red) with the nerve construction (blue). The category $\sSet$ is the category of simplicial sets. In general, given any category $\calX$, a \emph{simplicial $\calX$-object} is a functor $x:\Delta^{op} \to \calX$, i.e., an object of the functor category $\Functor{\Delta^{op}}{\calX}$. $\sSet$ corresponds to the special case when $\calX=\SetCat$, the category of small sets. The first ingredient is the Nerve functor $\Nerve_F$ (blue) which we have defined previously in \eqref{eqn:def:Nerve}. It encodes category $\calC$ into the category $\sSet$. Thereby every object in $\calC$ is assigned a combinatorial identity. 

The functor $\text{SimpHom}_{\calG}$ (red) converts this combinatorial information into algebraic form. It formalizes the idea of the number of holes of a given dimension in the complex. Every simplicial object is assigned a sequence of Abelian groups, and the $n$-th group of this sequence is a characterization of the number of $n$-dimensional holes. The construction of simplicial homology is completely independent of the category $\calC$ or the nerve construction. The details will be omitted and the readers are referred to classical sources such as \cite{Hatcher2002algtopo, Weibel1994Hmlgy, Rotman2013algtopo}. Broadly, it is the composite of three functors. This first involves taking a tensor product with $\calG$. This operation converts any simplicial set into a simplicial Abelian group. The face maps of the simplex $\Delta$ are reversed in direction and exist as \emph{co-face} maps in $\left[ \text{sGroup} \right]$. These are combined in an alternating sum to create an object called \textit{chain complex}. 

A chain complex consists of a sequence of Abelian groups $\paran{ C_n }_{n\in\num_0}$ along with group homomorphisms $\partial_n : C_n \to C_{n-1}$ such that $\partial_{n-1} \circ \partial_{n} \equiv 0$ for every index $n$. The homomorphisms are called \emph{boundary} maps. Element-wise maps between two chain complexes which commute with the boundary maps are called \emph{chain maps}. Chain complexes and chain maps together create the category $\ChainCmplx$. 

When $\calG=\integer$, the $n$-th chain complex object created out of the construction in \eqref{eqn:CatHmlgy} is the free Abelian group generated by $ \Hom_{\calC} \paran{ F(n); X } $. Given a typical chain complex as in 
\[\begin{tikzcd}
0 & C_0 \arrow[l, "0"] & C_1 \arrow[l, "\partial_1"] & \ldots \arrow[l, "\partial_2"] & C_{n-} \arrow[l, "\partial_{n-1}"] & C_n \arrow[l, "\partial_{n}"] & C_{n+1} \arrow[l, "\partial_{n+1}"] & \ldots \arrow[l, "\partial_{n+2}"]
\end{tikzcd}\]
let us set $B_n$ to be the image of $\partial_{n+1}$ and $Z_n$ to be the kernel of $\partial_n$. They are both subgroups of $c_n$ are known as the $n$-th \emph{boundaries} and \emph{cycles} respectively. By design, $B_n$ is a subgroup of $Z_n$. The $n$-th homology group of this chain complex is the quotient $Z_n/B_n$. This construction is also functorial and creates the last link in the construction of simplicial homology. 

This completes a very general discussion of homology, all stemming from the existence of the functor \eqref{eqn:functorF}. In summary, a homology functor embeds the category $\calC$ into the algebraic category of Abelian groups. Due to the simplistic nature of $\AbelCat$, it elucidates many properties within the category $\calC$. Consider the arrangement on the left below  :
\[\begin{tikzcd} 
& Y \arrow[dr, "g"] \\
X \arrow[ur, "f"] \arrow[rr, dashed, "\Id"', "\cong"] & & X
\end{tikzcd} ,\; 
\begin{tikzcd} {} \arrow[rr, mapsto] && {} \end{tikzcd} \; 
\begin{tikzcd} 
& \begin{array}{c} \Homology_k(Y) \\ =0 \end{array} \arrow[dr, "g_*"] \\
\Homology_k(X) \arrow[ur, "f_*"] \arrow[rr, dashed, "\Id"', "\cong"] & & \Homology_k(X)
\end{tikzcd}\]
The morphisms $g$ and $f$ are called a \emph{retraction} and \emph{section} respectively, and the object $X$ is said to be a retract  of $Y$. The diagram on the right is the image of the diagram on the left under the homology functor, evaluated at some index $k$. Suppose $\Homology_k(Y)=0$, as indicated. Since the intermediate group is zero, the composite group homomorphism map must also be zero. The functorial property of Homology dictates that the identity morphisms $\Id_{X}$ be mapped into a group isomorphism. Thus this diagram would be impossible if the $k$-th homology of $X$ is non-zero. A particular instance of such an argument is :

\begin{corollary} \label{cor:s9lx3}
	A non-acyclic object cannot be a retract of an acyclic object.
\end{corollary}

Corollary \ref{cor:s9lx3} happens to be a well known generalization of Brower's fixed point theorem \cite{Bredon2013topo}. 	We next consider the notion of Homotopy arising from this same functor.

%-_-_-_-_-_-_-_-_-_-_-_-_-_-_-_-_-_-_-_-_-_-_-_-_-_-_-_-_-_-_-_-_-_-_-_-_-_-_-_-_-_-_-_-_-_-_-_-_-_-_-_-_-_-_-_-_-_-_-_-_-_-_-_-_-_-_-_-_-_-_-_-_-_-_-_-
\section{Homotopy} \label{sec:hmtpy}

\begin{figure}
\centering	
\begin{tikzpicture}[scale=1.0, transform shape, framed, background rectangle/.style={double, ultra thick, draw=gray, rounded corners}]
	\input{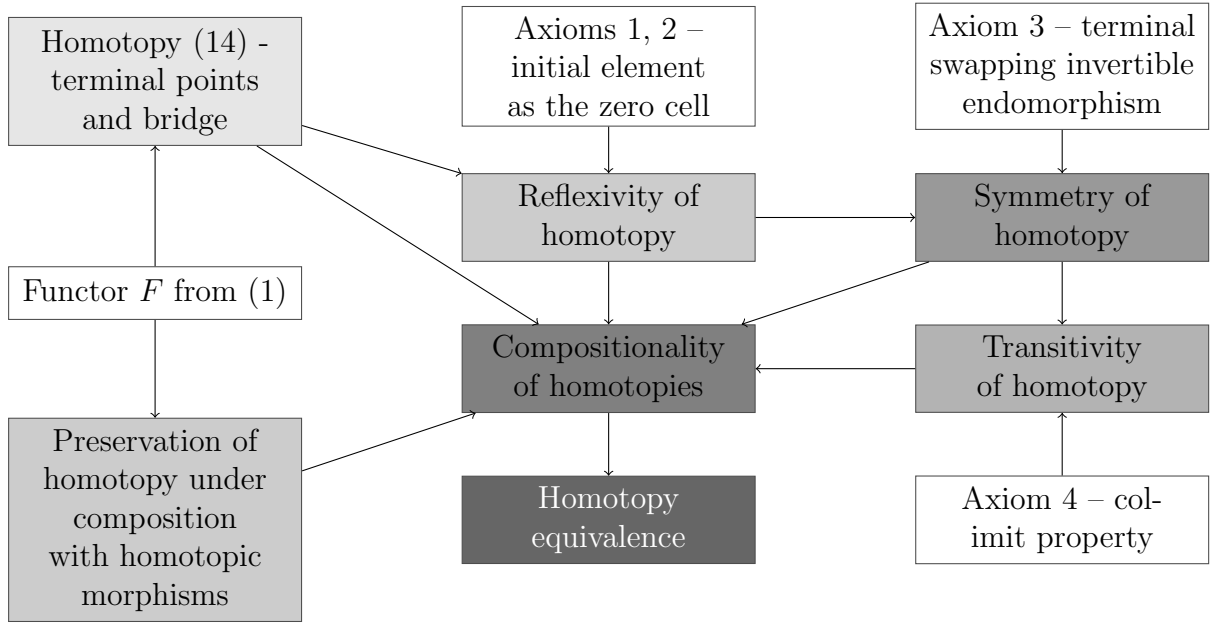}
\end{tikzpicture}
\caption{Construction of homotopy.}
\label{fig:hmtpy_outline}
\end{figure}

%In this section we assume that $F$ from \eqref{eqn:functorF} satisfies the following assumptions

The immediate consequence of Axioms \ref{A:C} and \ref{A:1_0_cell} is the following diagram :
\[\begin{tikzcd}
F(0) \arrow[rr, "F\paran{ d_{1,0} }"] && F(1) && F(0) \arrow[ll, "F\paran{ d_{1,1} }"'] \\
1_{\calC} \arrow[urr, dashed] \arrow[u, "="] && && 1_{\calC} \arrow[u, "="'] \arrow[ull, dashed]
\end{tikzcd}\]
This diagram also serves as the basis for the notion of homotopy. Given two objects $X,Y\in ob(\calC)$, we say that two morphisms $f,g:X\to Y$ are \emph{homotopic} if there is a morphism $H : X\times F(1) \to Y$ such that the following commutation holds :
\begin{equation} \label{eqn:def:hmtpy:2}
\begin{tikzcd} 
	X\times 1_{\calC} \arrow[Holud]{rrr}{ X \times F\paran{ d_{1,0} } } &&& X\times F(1) \arrow[Shobuj]{d}{H} &&& X\times 1_{\calC} \arrow[Akashi]{lll}[swap]{ X \times F\paran{ d_{1,1} } } \\ 
	X \arrow[Holud]{rrr}[swap]{f} \arrow[Holud]{u}{\cong} &&& Y &&& X \arrow[Akashi]{u}[swap]{\cong} \arrow[Akashi]{lll}{g}
\end{tikzcd}
\end{equation}
For the example of \eqref{eqn:def:StndrdTopoSmplx}, products are Cartesian products, so the object $X\times F(1)$ is just a cylinder whose axis is parameterized by the 1-simplex $F(1)=I$, and whose base or cross-section is the topological space $X$. In that case the homotopy $H$ according to \eqref{eqn:def:hmtpy:2} is a continuous map $H$ from the cylinder into $Y$ such that when restricted to the end faces of the cylinder, $H$ becomes $f$ and $g$ respectively. 

Note that out definition \eqref{eqn:def:hmtpy:2} itself does not assume its objects to be topological spaces or even collection of points. Neither does it assume anything about the morphisms other than its minimalist interpretation as an arrow. The notion of homotopy that emerges is diagrammatic, homotopy is merely an arrow that creates the diagram \eqref{eqn:def:hmtpy:2}. This definition does not imply a continuous transformation of continuous maps. It is only when we consider the category to be the special instance of Topology that the usual notion of homotopy take form. We next establish all the usual properties of Homotopy, while staying true to this minimalist diagrammatic interpretation which is free of notions such as points and continuity.

The relation of being homotopic is reflexive. Consider the diagram on the left below, in which the commutation holds because of the terminal property of $1_{\calC}$ and Axiom \ref{A:1_0_cell}. Taking a product of this diagram with $X$ gives the middle commuting diagram :
\[\begin{tikzcd} [scale cd = 0.7]
F(0) \arrow[dr, "\cong"', dashed] \arrow[r, "F\paran{d_{1,j}}"] & F(1) \arrow[d, "!"] \\
& F(0)
\end{tikzcd}
\; \begin{tikzcd} {} \arrow[mapsto, rr, "X\times"] && {} \end{tikzcd} \; 
\begin{tikzcd} [scale cd = 0.7]
X\times F(0) \arrow[d, "\cong"'] \arrow[dr, "\cong"', dashed] \arrow[r, "F\paran{d_{1,j}}"] & X\times F(1) \arrow[d, "!"] \\
X & X\times F(0) \arrow[l, "\cong"]
\end{tikzcd}
\; \begin{tikzcd} {} \arrow[mapsto, rr, "f\circ"] && {} \end{tikzcd} \; 
\begin{tikzcd} [scale cd = 0.7]
X\times F(0) \arrow[d, "\cong"'] \arrow[dr, "\cong"', dashed] \arrow[r, "F\paran{d_{1,j}}"] & X\times F(1) \arrow[d, "!"] \arrow[bend left=60, dd, dashed, "\bar{f}"] \\
X \arrow[dr, "f"'] & X\times F(0) \arrow[l, "\cong"] \\
& Y
\end{tikzcd}\]
In the rightmost diagram, we have appended the morphism $f:X\to Y$ to the middle diagram. Note that this commutation holds for $j=0,1$. This information can be redrawn as 
\[\begin{tikzcd} [column sep = large]
X\times 1_{\calC} \arrow{r}{ X \times F\paran{ d_{1,0} } } & X\times F(1) \arrow{d}{\bar{f}} & X\times 1_{\calC} \arrow{l}[swap]{ X \times F\paran{ d_{1,1} } } \\ 
X \arrow{r}[swap]{f} \arrow{u}{\cong} & Y & X \arrow{u}[swap]{\cong} \arrow{l}{f}
\end{tikzcd}\]
This diagram is a special case of \eqref{eqn:def:hmtpy:2} with $f=g$. This proves that $f=g$.

\paragraph{Symmetry} At present, homotopy is just a relation within each Hom-set $\Hom(X;Y)$. By virtue of Axiom \ref{A:swap}, the relation of homotopy turns out to be symmetric. This is because of the following commutation
\[\begin{tikzcd} [column sep = large]
& X\times I \arrow[d, "X \times \text{rev}" ] \arrow[Shobuj, bend right=30, dashed]{dd} \\
X\times \star \arrow{r}{ X \times F\paran{ d_{1,0} } } \arrow[bend left=20]{ur}{ X \times F\paran{ d_{1,1} } } & X\times I \arrow{d}{H} & X\times \star \arrow{l}[swap]{ X \times F\paran{ d_{1,1} } } \arrow[bend right=20]{ul}[swap]{ X \times F\paran{ d_{1,0} } } \\ 
X \arrow{r}[swap]{f} \arrow{u}{\cong} & Y & X \arrow{u}[swap]{\cong} \arrow{l}{g}
\end{tikzcd}\]
Thus $f$ being homotopic to $g$ is equivalent to saying that $g$  is homotopic to $f$. Homotopy is this a reflexive and symmetric property.

\paragraph{Transitivity} The limiting co-cone assumed in Axiom \ref{A:F1_join} is drawn below 
\begin{equation} \label{eqn:A:3}
\begin{tikzcd}
	&& F(1) \\
	F(1) \arrow[urr, bend left=30, Shobuj, "\text{left}"] && F(0) \arrow[rr, "F\paran{d_{0,0}}"'] \arrow[ll, "F\paran{d_{0,1}}"] \arrow[u, "\text{center}", Shobuj] && F(1) \arrow[ull, bend right=30, Shobuj, "\text{right}"']
\end{tikzcd}
\end{equation}
The connecting morphisms of the co-cone has been shown in green, and assigned some distinguishing names. This diagram will be the key to proving transitivity. Let there be three morphisms $f, g, h :X\to Y$ such that $f$ is homotopic to $g$ via a homotopy $H$, and $g$ is homotopic to $h$ via a homotopy $H$. To establish transitivity, we first take a product of the diagram in \eqref{eqn:A:3} with $X$ to get the diagram :
\[\begin{tikzcd} [column sep = large]
&& X\times F(1) \\
X\times F(1) \arrow[urr, bend left=30, "X\times \text{left}"] && X \arrow[rr, "X\times F\paran{d_{0,0}}"'] \arrow[ll, "X\times F\paran{d_{0,1}}"] \arrow[u, "\text{center}"] && X\times F(1) \arrow[ull, bend right=30, "X\times \text{right}"']
\end{tikzcd}\]
We now attend the morphisms and homotopies we assume in a second row below this diagram to get :
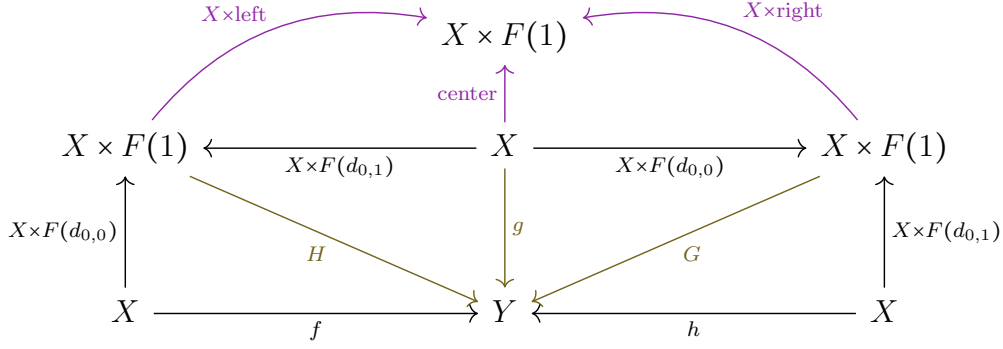
\[\begin{tikzcd} [column sep = large]
&& X\times F(1) \\
X\times F(1) \arrow[ddrr, "H"', Holud] \arrow[urr, bend left=30, "X\times \text{left}", Akashi] && X \arrow[dd, Holud, "g"] \arrow[rr, "X\times F\paran{d_{0,0}}"'] \arrow[ll, "X\times F\paran{d_{0,1}}"] \arrow[u, "\text{center}", Akashi] && X\times F(1) \arrow[ull, bend right=30, "X\times \text{right}"', Akashi] \arrow[ddll, "G", Holud] \\
\\
X \arrow[uu, "X\times F\paran{d_{0,0}}"] \arrow[rr, "f"'] && Y && X \arrow[uu, "X\times F\paran{d_{0,1}}"'] \arrow[ll, "h"]
\end{tikzcd}\]
Now observe that the morphisms colored in yellow are a co-cone above the diagram \eqref{eqn:dip3d}. Thus this cone must factorize through the limiting cone (blue), via a morphism shown in green :
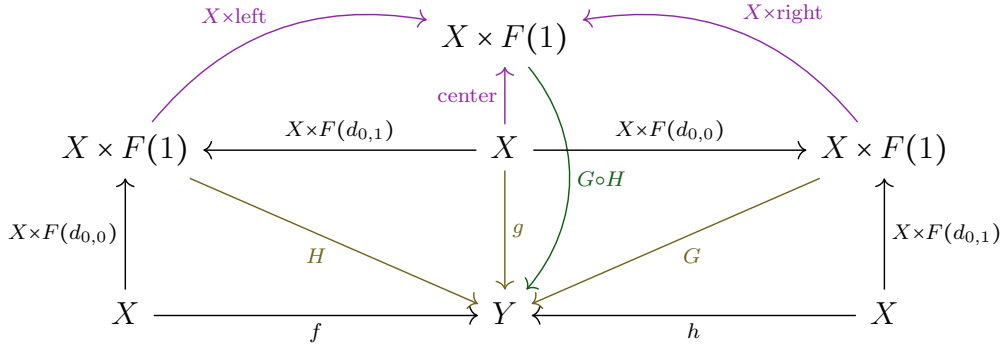
\[\begin{tikzcd} [column sep = large]
&& X\times F(1) \arrow[ddd, Shobuj, bend left=40, "G\circ H"] \\
X\times F(1) \arrow[ddrr, "H"', Holud] \arrow[urr, bend left=30, "X\times \text{left}", Akashi] && X \arrow[dd, Holud, "g"] \arrow[rr, "X\times F\paran{d_{0,0}}"] \arrow[ll, "X\times F\paran{d_{0,1}}"'] \arrow[u, "\text{center}", Akashi] && X\times F(1) \arrow[ull, bend right=30, "X\times \text{right}"', Akashi] \arrow[ddll, "G", Holud] \\
\\
X \arrow[uu, "X\times F\paran{d_{0,0}}"] \arrow[rr, "f"'] && Y && X \arrow[uu, "X\times F\paran{d_{0,1}}"'] \arrow[ll, "h"]
\end{tikzcd}\]
This green morphism is called the concatenation / vertical composition of the homotopies $H,G$. Note that the it serves as a homotopy from $f$ to $h$. Thus homotopy is a symmetric relation. 

\paragraph{Composition} We next prove that the property of being homotopic is preserved under pre- or post- composition with a fixed morphism. Throughout we assume the homotopy $H$ from \eqref{eqn:def:hmtpy:2}. Now suppose there is a morphism $h:Y\to Z$. Appending the morphism $h$ to \eqref{eqn:def:hmtpy:2} gives :
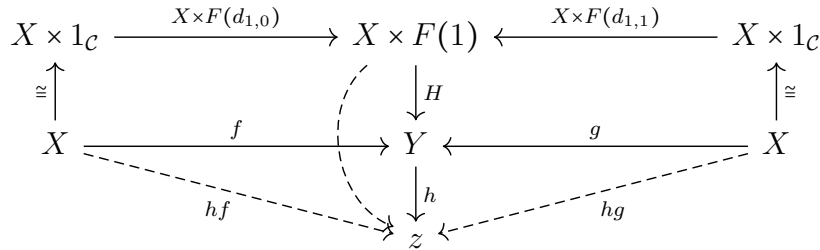
\[\begin{tikzcd} 
X\times 1_{\calC} \arrow{rrr}{ X \times F\paran{ d_{1,0} } } &&& X\times F(1) \arrow{d}{H} \arrow[dd, dashed, bend right=60] &&& X\times 1_{\calC} \arrow{lll}[swap]{ X \times F\paran{ d_{1,1} } } \\ 
X \arrow[drrr, dashed, "hf"'] \arrow{rrr}{f} \arrow{u}{\cong} &&& Y \arrow[d, "h"] &&& X \arrow{u}[swap]{\cong} \arrow{lll}[swap]{g} \arrow[dlll, dashed, "hg"] \\
&&& z
\end{tikzcd}\]
This creates a homotopy $hH$ which serves as a bridge between $hf$ and $hg$. Next we assume a morphism $h:W\to X$. Then the diagram in \eqref{eqn:def:hmtpy:2} can be expanded to :
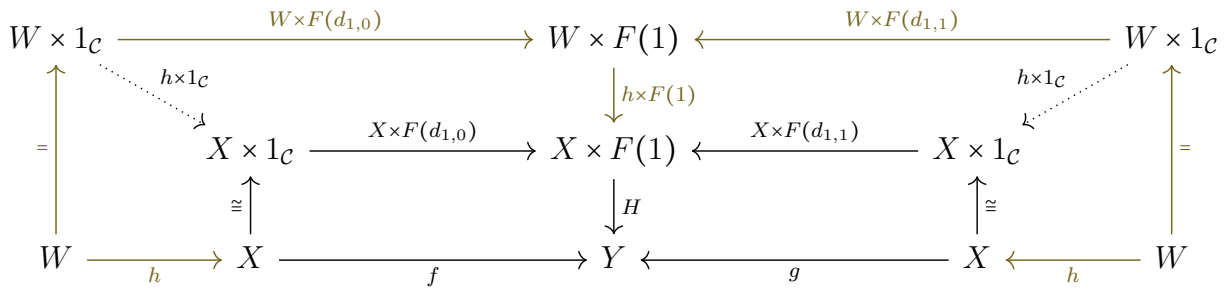
\[\begin{tikzcd} 
W\times 1_{\calC} \arrow[dr, dotted, "h\times 1_{\calC}"] \arrow[Holud]{rrrr}{ W \times F\paran{ d_{1,0} } } & &&& W\times F(1) \arrow[Holud]{d}{ h\times F(1) } &&& & W\times 1_{\calC} \arrow[Holud]{llll}[swap]{ W \times F\paran{ d_{1,1} } } \arrow[dl, dotted, "h\times 1_{\calC}"'] \\
& X\times 1_{\calC} \arrow{rrr}{ X \times F\paran{ d_{1,0} } } &&& X\times F(1) \arrow{d}{H} &&& X\times 1_{\calC} \arrow{lll}[swap]{ X \times F\paran{ d_{1,1} } } &  \\ 
W \arrow[uu, Holud, "="] \arrow[Holud, r, "h"'] & X \arrow{rrr}[swap]{f} \arrow{u}{\cong} &&& Y &&& X \arrow{u}[swap]{\cong} \arrow{lll}{g} & W \arrow[Holud, l, "h"] \arrow[uu, Holud, "="']
\end{tikzcd}\]
This creates a homotopy $H \circ \paran{ h\times F(1) }$ which serves as a bridge between $fh$ to $gh$. This proves that homotopy is preserved under pre- and post- composition. This bring us to the last and final property :

\paragraph{Composition of homotopies} Suppose there are three objects $X,Y,Z$, and homotopic pairs of morphisms $f,g : X\to Y$ and $f',g':Y\to Z$. Let these homotopies be $H,G$ respectively. Then note that
\[\begin{split}
f'\circ f &\sim f'\circ g , \quad \mbox{ post-composition of } H \mbox{ with } f' ,\\
&\sim g'\circ g , \quad \mbox{ pre-composition of } G \mbox{ with } g .
\end{split}\]
The $\sim$ denotes a homotopy relation in the Hom-set $\Hom(X;Z)$, whose transitivity has already been established. Thus we can conclude that $f'\circ f$ is homotopic to $g'\circ g$. This relation can be summarized by saying that homotopies can be composed.

\paragraph{Homotopy equivalence} A morphism in $\Hom(X;X)$ is said to be null-homotopic if it is homotopic to $X$. Two objects $A,B$ will be called \emph{homotopy equivalent} \cite{Riehl2014homotopy} if there are morphisms $f:A\to B$ and $g:B\to A$ such that $f\circ g \sim \Id_B$ and $g\circ f \sim \Id_A$. In other words, both $f\circ g$ and $g\circ f$ are null-homotopic.

Under Axioms \ref{A:C}, \ref{A:1_0_cell}, \ref{A:swap} and \ref{A:F1_join} the relation of being homotopic is symmetric, reflexive and transitive. Thus this relation can be extended to an equivalence relation on $\Hom_{\calC}(X;Y)$. One of the most important consequences of homotopy is that one can derive a new category $\tilde{\calC}$ from $\calC$, which has the same object-sets, but in which the Hom-set $\Hom_{\tilde\calC}(X;Y)$ are the homotopy equivalence classes of $\Hom_{\calC}(X;Y)$. This is a well defined operation which retains the compositionality of $\calC$. Note that homotopy equivalence in $\calC$ is the same as isomorphism in $\tilde\calC$. This completes the construction of the notion of homotopy. 

\begin{table}[!t]
\scriptsize
\caption{\textbf{Related work.}}
\begin{tabularx}{\linewidth}{|p{0.12\textwidth}|p{0.12\textwidth}|p{0.18\textwidth}|p{0.18\textwidth}|p{0.18\textwidth}|} \hline
	Theory & Unique foundational assumptions & Interpretation of Homotopy & Interpretation of Homology & Challenges \\ \hline 
	$\infty$-categories \cite{joyal2008notes, lurie2008what} & \emph{inner fibrations} of $\sSet$ & Homotopy Limits and Colimits appear naturally as weighted limits/colimits in a higher category & & \\ \hline 
	Model-categories \cite{Quillen2006hmtpy} & weak equivalences, (co)-fibrations & Associated \emph{homotopy category} $\text{Ho}(\calC)$ encodes the information of weak homotopy equivalences & & Reliance on point-set definitions;  \\ \hline 
	Classical generalized homotopy theory & & & & \\ \hline 
	Stable homotopy theory & & & & \\ \hline 
	\Shobuj{Present work} & Convexity as a natural transformation and inverse to a \textit{decalage} & Homotopy realized through an interval like object & Homology is the usual simplicial homology obtained via $\Nerve$-construction &  \\ \hline 
\end{tabularx}
\label{tab:compare:1}
\end{table}

\paragraph{Related work} Categorical homotopy theory has evolved through a series of increasingly abstract and powerful axiomatic frameworks, moving from algebraic topology's specific problems to a general theory of "homotopy" in any category. Some of the earliest generalizations was in the use of (pre)-sheaves Eilenberg and Mac Lane \cite{EilenbergMacLane1945hmtpy, EilenbergMacLane1950hmtpy}, and \emph{Kan Complexes} \cite{Kan1958css}. These developments were aimed at providing a purely combinatorial model for topological spaces, making homotopy more algebraic and easier to manipulate.

The first major abstraction of Homotopy was Quillen's work  \cite{Quillen2006hmtpy} which introduced the concept of a \emph{Model Category} to formalize the common structure found in the various homotopy theories. Like the present work, his axiomatization was set in a general category and had three main ingredients -- \emph{weak equivalence}s, \emph{fibrations} and \emph{co-fibrations}. His formulation of Homotopy recreates the familiar notions of Homotopy using the operations of \emph{factorization} and \emph{lifting}. The goal is to formally invert the weak equivalences to form a homotopy category $\text{Ho}(\mathcal{C})$ which captures the core "homotopical" information. One of the key utilities of Quillen's approach was the availability of \emph{Quillen Adjunctions}. These are adjoint functors $F: \mathcal{C} \rightleftarrows \mathcal{D}: G$ between model categories that preserve the homotopical structure, providing a systematic way to relate different homotopy theories. While model categories provided a powerful, uniform language, their \emph{point-set} nature and reliance on specific classes of maps (fibrations / cofibrations) led to technical complexities. 

The last phase of evolution of Homotopy theory focused on Higher Category Theory ($\infty$-categories or $(\infty, 1)$-categories) as the fundamental realm for homotopy theory. There are two main lines of abstraction using this method. The first is the idea of \emph{quasi-categories} by Joyal and Lurie \cite[e.g.]{joyal2008notes, lurie2008what} based on \emph{inner fibrations} of $\sSet$ provides a combinatorial model for $(\infty, 1)$-categories. This framework allows for the direct study of Homotopy Limits and Colimits as a natural consequence of weighted limits/colimits in a higher category, often simplifying the complex constructions required in model categories. The second distinct method is \emph{Homotopy Type Theory} (HoTT) which interprets types as \emph{homotopy types} (or $\infty$-groupoids), formalizing a foundation of mathematics where the "equality" of objects is viewed as a path (or higher path) in a space.

The present axiomatization avoids all this heavy axiomatic framework. It relies on the basic notion of an abstract interval by way of Axiom \ref{A:1_0_cell}. Figure \ref{fig:hmtpy_outline} presents an overview of how this and the other axioms lead to homotopy equivalences of morphisms, as well as notions of cyclicity. The present framework does not rely on points and sets and at the same time achieves sufficient generality. Note that the axioms \ref{A:C} -- \ref{A:F1_join} only involve a finite number of morphisms and objects ($F(0)$ and $F(1)$), and does not require an elaborate factorization or fibration system. This simplicity will be useful in the next Section, where the connections with Homology will be established.

%-_-_-_-_-_-_-_-_-_-_-_-_-_-_-_-_-_-_-_-_-_-_-_-_-_-_-_-_-_-_-_-_-_-_-_-_-_-_-_-_-_-_-_-_-_-_-_-_-_-_-_-_-_-_-_-_-_-_-_-_-_-_-_-_-_-_-_-_-_-_-_-_-_-_-_-
\section{Homotopy invariance of homology} \label{sec:hmlgy_htpy:2} 

\begin{figure}
	\centering
	\begin{tikzpicture}[scale=1.0, transform shape, framed, background rectangle/.style={double, ultra thick, draw=gray, rounded corners}]
		\input{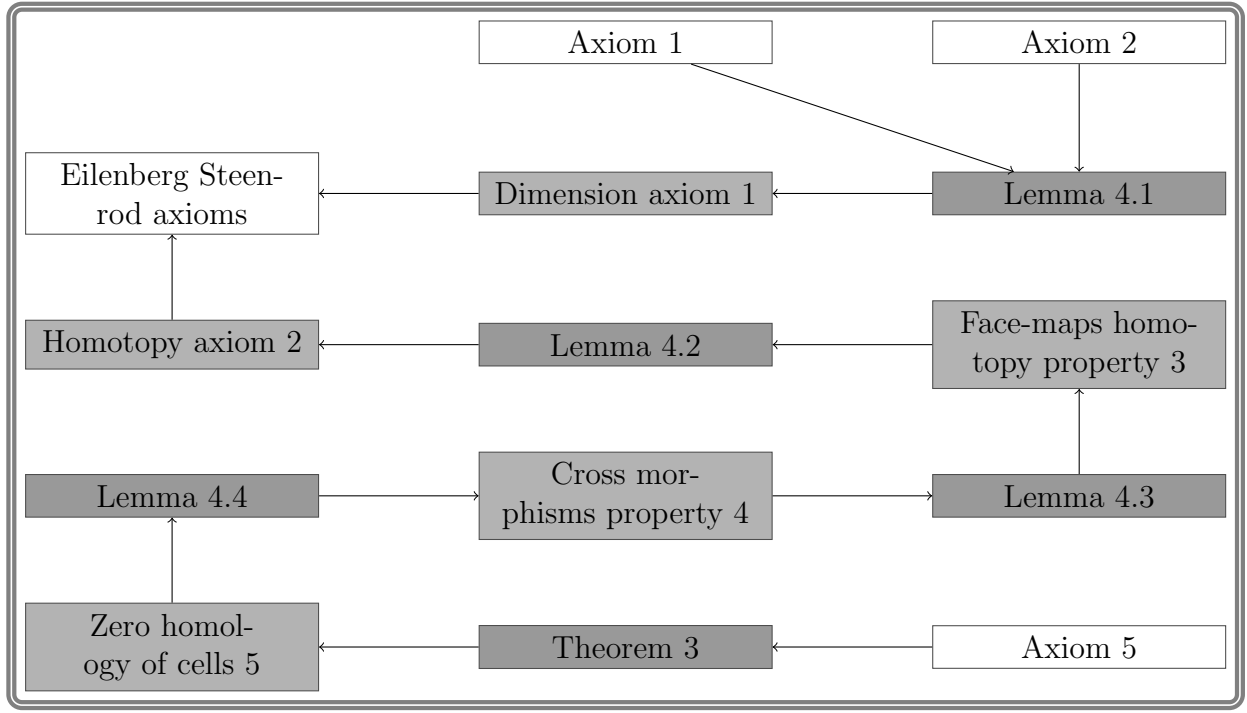}
	\end{tikzpicture}
	\caption{Axiomatic basis of the homotopy invariance of homology.}
	\label{fig:hmlgy_hmtpy}
\end{figure} 

We now examine the axioms and mechanisms which guarantee that homotopic maps induce the same homomorphisms between homology groups. We shall follow the outline presented in Figure \ref{fig:hmlgy_hmtpy}. One of the basic properties to desire from a homology theory is 

\begin{property}[Dimension axiom] \label{P:1}
	Homology of $F(0)$ is zero for indices $n>0$.
\end{property}

This property is granted to us by Axiom \ref{A:1_0_cell} :

\begin{lemma} [Dimension axiom] \label{lem:A1_to_P1}
	If the functor from \eqref{eqn:functorF} satisfies Axiom \ref{A:1_0_cell}, then the homology it induces satisfies the dimension axiom \ref{P:1}.
\end{lemma}

The lemma follows from the fact that $F(0)$ is the terminal object of $\calC$. Thus for each $n\in \num$ and $c\in \calC$, $\Nerve_F(c)(n) = \Hom_{\calC} \paran{ F(n); 1}$ is a singleton set. The proof is algebraic and can be found in \cite[Thm 4.12]{Rotman2013algtopo}. The next property is the main focus of this chapter as well as this article : 

\begin{property}[Homotopy axiom] \label{P:3}
	The notion of homology is homotopy invariant. In other words, if $f,g : X\to Y$  are homotopic, then their induced homologies are the same.
\end{property}

The homotopy axiom can be derived from a simpler property involving the boundary maps :

\begin{property} [Face-maps homotopy property] \label{P:4}
	An object $X\in \calC$ is said to have this property if the following composite morphisms in $\calC$
	\[\begin{tikzcd}
		&& X \arrow[d, "\cong"] \arrow[dashed, drr, "\lambda^{X}_0"] \arrow[dll, dashed, "\lambda^{X}_1"'] \\
		X\times F(1) && X\times \star \arrow[rr, "\Id_x \times F\paran{d_{1,0}}"'] \arrow[ll, "\Id_x \times F\paran{d_{1,1}}"] && X\times F(1)
	\end{tikzcd}\]
	induce the same morphisms under Homology.
\end{property}

As outlined in Figure \ref{fig:hmlgy_hmtpy} we first connect Property \ref{P:3} to Property \ref{P:4} :

\begin{lemma} \label{lem:P4_to_P3}
	\cite[Lem 4.20]{Rotman2013algtopo} Suppose each object $X$ has Property \ref{P:4}. Then the Homotopy axiom \ref{P:3} is satisfied. 
\end{lemma}

\begin{proof} For the proof, we shall make use of the $\Chain$-functor displayed in Diagram \eqref{eqn:CatHmlgy}. For brevity we shall denote it as $\Chain$ and drop its subscript from the notation.  Recall the morphisms $\lambda^{X}_j$ from Property \ref{P:4}.   A homotopy between morphisms $f,g:X\to Y$ is a morphism $H:X\times F(1) \to Y$ such that $f = H\circ \lambda^{X}_0$ and $g = H\circ \lambda^{X}_1$. Then note that
	\[\begin{split}
		\Chain\paran{ f }_{\#} &= \Chain\paran{H\circ \lambda^{X}_0}_{\#} = \Chain\paran{ H }_{\#} \circ \Chain\paran{ \lambda^{X}_0 }_{\#} \\
		& = \Chain\paran{ H }_{\#} \circ \Chain\paran{ \lambda^{X}_1 }_{\#} , \quad \mbox{by Property \ref{P:4}} , \\
		&= \Chain\paran{H\circ \lambda^{X}_1}_{\#} = \Chain\paran{g}_{\#} .
	\end{split}\]
	In the equalities above, we have repeatedly used the composition preserving property of the functor $\Chain$. Since $f,g$ are mapped under the $\Chain$ functor into the same morphisms in the category of chain complexes, they also remain equal under application of the Chain-homology functor. %Thus $f,g$ induce the same homology maps too.
\end{proof}

Thus to prove Homotopy property \ref{P:3}, we are presented with the task of proving Property \ref{P:4}. As outlined in Figure \ref{fig:hmlgy_hmtpy} we connect Property \ref{P:4} to Property \ref{P:5} :

\begin{property} [Cross morphisms property] \label{P:5}
	An object $X\in \calC$ has this property if there are homomorphisms
	\[ P^{X}_{n} : \Chain(X)_n \to \Chain\paran{ X\times F(1) }_{n+1} \]
	such that
	\begin{equation} \label{eqn:P:5a}
		\Chain\paran{ \lambda^{X}_1 }_n - \Chain\paran{ \lambda^{X}_0 }_n = \partial_{n+1} P^{X}_{n} + P^{X}_{n-1} \partial_{n} .
	\end{equation}
\end{property}

\begin{lemma} \label{lem:P5_to_P4}
	\cite[Thm 4.23]{Rotman2013algtopo} For any object $X\in \calC$, Property \ref{P:5} implies Property \ref{P:4}.
\end{lemma} 

\begin{proof} Let $\Boundary(X)_n$ denote the subgroup of $\Chain(X)_n$, created by boundaries. Let $z$ be a cycle in $\Chain(X)_n$, i.e., $\partial_n z = 0$. By the definition of the homology maps 
	\[ \Homology(f)_n : z \oplus \Boundary(X)_n \mapsto \Chain(f)_n (z) + \Boundary(Y)_n \]
	But by Property \ref{P:5}, we have
	\[ \paran{ \Chain(f)_n - \Chain(g)_n } (z) = \partial^Y_{} P z + P \partial z  = \partial^Y_{} P z . \]
	Since the RHS is a boundary element, so is the LHS. As a result, we can write
	\[ \Chain(f)_n(z) \oplus \Boundary(Y)_n = \Chain(g)_n(z) \oplus \Boundary(Y)_n . \]
	This means that $\Homology(f)_n$ must be equal to $\Homology(g)_n$. 
\end{proof}

Thus to prove Homotopy property \ref{P:3}, we are presented with the task of proving Property \ref{P:5}. As outlined in Figure \ref{fig:hmlgy_hmtpy} we connect Property \ref{P:5} to Property \ref{P:2} :

\begin{property}[Acyclic cells] \label{P:2}
	The cells $\SetDef{F(n)}{n\in\num_0}$ are acyclic, i.e., their $k$-th homology group is zero for every $k>0$.
\end{property}

%An object $X\in \calC$ will be called \emph{acyclic} if its homology group is zero. Thus Property \ref{P:2} says that the cells $F(n)$ are all acyclic.

Finally the acyclicity Property \ref{P:2} can be derived as a consequence of the following important result

\begin{theorem} [Convexity] \label{thm:cnvxty}
	%Let Axioms \ref{A:C}, \ref{A:1_0_cell}, \ref{A:swap}, \ref{A:F1_join} and \ref{A:convex} hold, and recall the definition of convexity. 
	Suppose Axioms \ref{A:C} and \ref{A:1_0_cell} hold.  Then the collection of convex objects of $\calC$
	\begin{enumerate} [(i)]
		\item are closed under finite products;
		\item are acyclic.
	\end{enumerate}
\end{theorem}

Theorem \ref{thm:cnvxty} is proved in Section \ref{sec:cnvxty}. Also see Figure \ref{fig:cnvxty} for an outline of the chain of reasoning. Our last lemma in this section which completes the chain of reasoning outlined in Figure \ref{fig:hmlgy_hmtpy} is :

\begin{lemma} \label{lem:P2_to_P5}
	Suppose Axioms \ref{A:C}, \ref{A:1_0_cell} and \ref{A:convex} hold. Then Property \ref{P:2} implies Property \ref{P:5} for every object $X\in \calC$.
\end{lemma}

Lemma \ref{lem:P2_to_P5} is an outcome of the theory of \emph{acyclic models} \cite[]{EilenbergMacLane1953acyclic, Barr1996acyclic}. We provide the proof in Section \ref{sec:proof:P2_to_P5}, for the sake of completeness.

\paragraph{Verifying Property \ref{P:3}} By Axiom \ref{A:convex}, the cells $\SetDef{ F(n) }{ n\in\num_0 }$ are convex. Thus by Theorem \ref{thm:cnvxty} (ii), they are acyclic too, as claimed in Property \ref{P:2}. This completes the verification of the Homotopy axiom. \qed 

\paragraph{Remark} Similar to classic Homotopy theory, an object $X$ of $\calC$ will be called \emph{contractible} if its identity morphism $X$ is homotopic to some constant endomorphism. A \emph{constant endomorphism} is a composite of morphisms of the form $X \xrightarrow{!} 1_{\calC} \xrightarrow{x} X$. Here $x$ is an element of $X$ and the composite morphism can be interpreted to be constant of value $x$. Theorems \ref{thm:1}~(iii) and \ref{thm:cnvxty}~(ii) together imply that contractible objects are acyclic.

%-_-_-_-_-_-_-_-_-_-_-_-_-_-_-_-_-_-_-_-_-_-_-_-_-_-_-_-_-_-_-_-_-_-_-_-_-_-_-_-_-_-_-_-_-_-_-_-_-_-_-_-_-_-_-_-_-_-_-_-_-_-_-_-_-_-_-_-_-_-_-_-_-_-_-_-
\section{Convexity} \label{sec:cnvxty}

\begin{figure}
	\centering
	\begin{tikzpicture}[scale=0.7, transform shape]
		\input{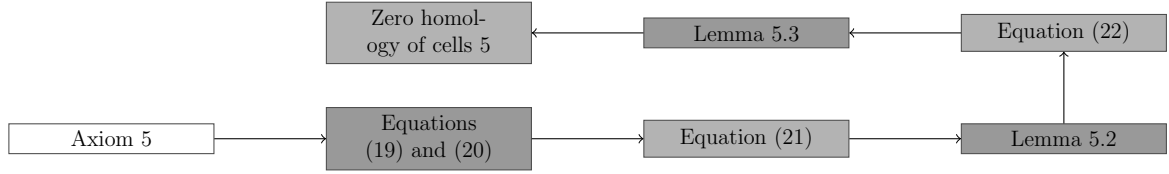}
	\end{tikzpicture}
	\caption{Convexity and acyclicity. The flowchart is an outline of the proof of Theorem \ref{thm:cnvxty}~(ii).}
	\label{fig:cnvxty}
\end{figure}

The main purpose of this section is to examine the consequences of convexity by verifying the claims of Theorem \ref{thm:cnvxty}. A major part of the discussion that follows is outlined in Figure \ref{fig:cnvxty}. Convexity is borne by the natural transformations in \eqref{eqn:def:Base} and \eqref{eqn:def:ConeNat}. The first one indicates the following commutations : 
\begin{equation} \label{eqn:Base:1}
	\begin{tikzcd}
		\Hom_{ \calC } \paran{ F(n+1); X  } \arrow{d}[swap]{ \circ F \paran{ d_{n+1,0} } } \arrow{rrr}{ \circ F \paran{ d_{n+1,i+1} } } &&& \Hom_{ \calC } \paran{ F(n); X  } \arrow{d}{ \circ F \paran{ d_{n,0} } } \\ 
		\Hom_{ \calC } \paran{ F(n); X  } \arrow{rrr}{ \circ F \paran{ d_{n,i} } } &&& \Hom_{ \calC } \paran{ F(n-1); X  }
	\end{tikzcd} , \quad \forall 0\leq i \leq n .
\end{equation}
The commutation in \eqref{eqn:Base:1} follows from \eqref{eqn:smplc_id}~(i). Thus the $\Base$ transformation exists and is independent of $X$. We next examine the consequences of \eqref{eqn:def:ConeNat}. It implies the existence of a special collection of maps
\begin{equation} \label{eqn:cone:a}
	\Cone^X_n : \Hom_{\calC} \paran{ F(n); X } \to \Hom_{\calC} \paran{ F(n+1); X } , \quad n\in \num_0 .
\end{equation}
which satisfy 
\begin{equation} \label{eqn:ConeNat:3}
	\begin{tikzcd}
		\Hom_{ \calC } \paran{ F(n); X  } \arrow[d, "\Cone^{X}_n"'] \arrow{rrr}{ \circ F \paran{ d_{n,i} } } &&& \Hom_{ \calC } \paran{ F(n-1); X  } \arrow[d, "\Cone^{X}_{n-1}"] \\ 
		\Hom_{ \calC } \paran{ F(n+1); X  } \arrow{rrr}{ \circ F \paran{ d_{n+1,i+1} } } &&& \Hom_{ \calC } \paran{ F(n); X  }
	\end{tikzcd} , \quad 
	0\leq i\leq n
\end{equation}
and
\begin{equation} \label{eqn:Base:2}
	\begin{tikzcd}
		\Hom_{\calC} \paran{ F(n); X } \arrow{d}[swap]{ \Cone^X(n) } \arrow[r, "="] & \Hom_{\calC} \paran{ F(n); X } \\ 
		\Hom_{\calC} \paran{ F(n+1); X } \arrow[ur, "\circ F\paran{ d_{n+1,0} }"']
	\end{tikzcd},
	\quad \forall n\in\num_0 .
\end{equation}
The set maps $\Cone^X_n$ from \eqref{eqn:cone:a} extend uniquely to homomorphisms on the freely generated Abelian groups 
\[ \Cone^X_n : \Chain \paran{ X}_n \to \Chain \paran{ X}_{n+1} ; \quad \forall n\in\num_0 . \]
Now observe that \eqref{eqn:ConeNat:3} and \eqref{eqn:Base:2} together imply 
\begin{equation} \label{eqn:P:6}
	\forall \begin{tikzcd} F(n) \arrow[ "\sigma"', d ] \\ X	\end{tikzcd} , \quad 
	\Cone^{X}_n \paran{ \sigma } \circ F\paran{ d_{n+1,i} } \,=\, 
	\begin{cases}
		\Cone^{X}_{n-1} \paran{ \sigma \circ F \paran{ d_{n,i-1} } } & \mbox{ if } 1\leq i \leq n+1 \\
		\sigma & \mbox{ if } i=0
	\end{cases} ; \quad \forall n\in \num_0 .
\end{equation}
The following lemma is a summary of the various discussions on convexity we have had so far in Section \ref{sec:cnvxty} :

\begin{lemma} [Equivalent characterizations of convexity] \label{lem:cdi9}
	Let Axioms \ref{A:C} and \ref{A:1_0_cell} hold. Then the following are equivalent for any object $X$ 
	\begin{enumerate} [(i)]
		\item $X$ is convex. 
		\item $X$ has associated with a cone construction in the sense of \eqref{eqn:def:ConeNat}.  
		\item $X$ has a sequence of maps as in \eqref{eqn:cone:a}, such that the identities in \eqref{eqn:ConeNat:3} and \eqref{eqn:Base:2} are satisfied.
		\item $X$ has a sequence of maps as in \eqref{eqn:cone:a}, such that the identities in \eqref{eqn:P:6} are satisfied. 
	\end{enumerate}
\end{lemma}	

We can now analyze the first consequence of convexity directly at the level of homology. Note that we have 

\begin{lemma} \label{lem:P6toP7}
	Let Axioms \ref{A:C} and \ref{A:1_0_cell} hold. Suppose that an object $X$ has a sequence of maps $\Cone_n$ as in \eqref{eqn:cone:a}, thus leading to the following parallel pair of morphisms
	\[\begin{tikzcd}
		\Chain(X)_n \arrow[d, "\Psi_n"'] \arrow[r, dashed] & \Chain(X)_n \\
		\Chain(X)_{n+1} \arrow[ur, "\partial^X_{n+1}"']
	\end{tikzcd} ,\quad 
	\begin{tikzcd}
		\Chain(X)_n \arrow[d, "\partial^X_n"'] \arrow[r, dashed] & \Chain(X)_{n-1} \\
		\Chain(X)_{n+1} \arrow[ur, "\Psi_{ni1}"']
	\end{tikzcd}\]
	Suppose that the maps $\Cone_n$ also satisfy \eqref{eqn:P:6}. Then 
	\begin{equation} \label{eqn:P:7}
		\partial_{n+1} \Cone_n + \Cone_{n-1} \partial_{n} = \Id_n; \quad \forall n\in \num .
	\end{equation}
	where $\Id_n$ is the identity morphism on $\Chain(X)_n$.
\end{lemma}

\paragraph{Remark} Note that the identity \eqref{eqn:P:7} is only claimed for $n>0$. This is because of the involvement of the boundary map $\partial_{n}$. This has important consequences for the homology groups, as will be shown later in Lemma \ref{lem:P7toP2}.

The lemma follows from the following simple algebraic calculation : 
\[\begin{split}
	\partial_{n+1} \paran{ \Cone_n \sigma } &= \sum_{i=0}^{n+1} (-1)^i \paran{ \Cone_n \sigma } \circ F \paran{ d_{n+1,i} } = \sigma + \sum_{i=1}^{n+1} (-1)^i \paran{ \Cone_n \sigma } \circ F \paran{ d_{n+1,i} } \\
	&= \sigma + \sum_{i=1}^{n+1} (-1)^i \Cone_{n-1} \paran{ \sigma \circ F \paran{ d_{n+1,i-1} } } = \sigma + \Cone_{n-1} \paran{ \sum_{i=1}^{n+1} (-1)^i \sigma \circ F \paran{ d_{n+1,i-1} } } \\
	&= \sigma - \Cone_{n-1} \paran{ \partial_{n} \sigma } . \qed
\end{split}\]
Equation \eqref{eqn:P:7} distills the essential identity from the natural transformations $\Base$ and $\Cone$ needed to guarantee acyclicity : 

\begin{lemma} \label{lem:P7toP2}
	Let Axioms \ref{A:C} and \ref{A:1_0_cell} hold. Suppose that an object $X$ has a sequence of maps $\Cone_n$ as in \eqref{eqn:cone:a}, such that their chain-map extensions satisfy \eqref{eqn:P:7}. Then $X$ is acyclic.
\end{lemma}

\begin{proof} It has to be shown that for every $n\in\num$ any cycle in $\Chain(X)_n$ is also a boundary. Take any cycle $\sigma\in \Chain(X)_n$ i.e., $\partial_{n} \sigma = 0$. Then according to \eqref{eqn:P:7}, $\sigma = \partial_{n+1} \Cone_n \sigma$. Thus $\sigma$ is a boundary, as required for the completion of the proof. 
\end{proof}

Thus we have established Theorem \ref{thm:cnvxty}~(ii) We next focus on Claim ~(i). For that purpose, suppose that $X$ and $Y$ are two convex objects, with cones $\Cone^X$ and $\Cone^Y$ respectively. Note that any morphism in the Hom-set $\Hom_{\calC} \paran{ F(n) ; X\times Y }$ is a unique categorical product of morphisms in $\Hom_{\calC} \paran{ F(n) ; Y }$ and $\Hom_{\calC} \paran{ F(n) ; X }$ respectively. Thus it is enough to prescribe $\Cone_n^{X\times Y}$ on product morphisms only. Consider the following commuting diagram : 	
\[ \forall \begin{tikzcd} F(n) \arrow[d, "\alpha"'] \\ X \end{tikzcd} ,\,
\forall \begin{tikzcd} F(n) \arrow[d, "\beta"'] \\ Y \end{tikzcd} \,\Rightarrow\,
\begin{tikzcd}
	&& && X\\
	&& F(n+1) \arrow{rru}{ \Cone^X_n (\alpha) } \arrow{rrd}[swap]{ \Cone^Y_n (\beta) } \arrow[Shobuj]{rrrr}{ \left\langle \Cone^X_n (\alpha), \Cone^Y_n (\beta) \right\rangle } && && X\times Y \arrow[ull, "\proj_1"'] \arrow[dll, "\proj_2"] \\
	&& && Y
\end{tikzcd}\]
The green arrow  is the product of two cones, and the commutations represent its relationship with its two coordinate projections. We are going to set 
\begin{equation} \label{eqn:ConeProd}
	\forall \alpha \in \Hom_{ \calC } \paran{ F(n); X }, \; \forall \beta \in \Hom_{ \calC } \paran{ F(n); Y }, \quad \Cone_n^{X\times Y} \paran{ \left\langle \alpha, \beta \right\rangle } = \left\langle \Cone_n^X(\alpha) , \Cone_n^Y (\beta) \right\rangle , \quad \forall n\in \num_0 .
\end{equation}
To check that this is indeed a cone construction for $X\times Y$, \eqref{eqn:ConeNat:3} has to be verified. For that purpose, we append a face map to the left of the commutation diagram to get :
\[\begin{tikzcd}
	&& && X\\
	F(n) \arrow[rr, "F(d_{n+1,i})", Shobuj] \arrow[Holud, dashed, bend left=20, urrrr] \arrow[Holud, dashed, bend right=20, drrrr] && F(n+1) \arrow{rru}{ \Cone^X_n (\alpha) } \arrow{rrd}[swap]{ \Cone^Y_n (\beta) } \arrow[Shobuj]{rrrr}{ \left\langle \Cone^X_n (\alpha), \Cone^Y_n (\beta) \right\rangle } && && X\times Y \arrow[ull, "\proj_1"'] \arrow[dll, "\proj_2"] \\
	&& && Y
\end{tikzcd}\]
The dashed yellow arrows represent compositions. Thus the composite of the green arrows is the product of the yellow arrows. Now if $i=0$ then by \eqref{eqn:P:6} the yellow arrows are $\alpha$ and $\beta$ respectively. This means that 
\[ \left\langle \Cone_n^X(\alpha) , \Cone_n^Y (\beta) \right\rangle \circ F \paran{ d_{n+1,0} } = \left\langle \alpha , \beta \right\rangle .\]
If $1\leq i \leq n+1$, then again by \eqref{eqn:P:6} the yellow arrows are $\Cone^{X}_{n-1} \paran{ \alpha \circ F \paran{ d_{n,i-1} } }$ and $\Cone^{Y}_{n-1} \paran{ \beta \circ F \paran{ d_{n,i-1} } }$ respectively. This means that 
\[\begin{split}
	& \left\langle \Cone_n^X(\alpha) , \Cone_n^Y (\beta) \right\rangle \circ F \paran{ d_{n+1,i} } = \left\langle \Cone^{X}_{n-1} \paran{ \alpha \circ F \paran{ d_{n,i-1} } } , \Cone^{Y}_{n-1} \paran{ \beta \circ F \paran{ d_{n,i-1} } } \right\rangle \\
	& \quad \quad = \Cone^{X}_{n-1} \paran{ \left\langle \alpha \circ F \paran{ d_{n,i-1} } , \beta \circ F \paran{ d_{n,i-1} } \right\rangle }, \quad \mbox{by definition}, \\
	& \quad \quad =  \Cone^{X}_{n-1} \paran{ \left\langle \alpha , \beta \right\rangle } \circ F \paran{ d_{n,i-1} } .
\end{split}\]
Thus \eqref{eqn:ConeProd} indeed provides a cone construction for $X\times Y$. So $X\times Y$ is convex, and Claim (ii) is true. This completes the proof of Theorem \ref{thm:cnvxty}. \qed 

This completes the discussion on how convexity implies acyclicity. Thus all the ingredients of Theorem \ref{thm:1} are in place. While it is easy to conceive of functors \eqref{eqn:functorF}, it is a non-trivial task to satisfy its axioms. Axiom \ref{A:convex} is the most challenging task to be verified for a general functor. 

%-_-_-_-_-_-_-_-_-_-_-_-_-_-_-_-_-_-_-_-_-_-_-_-_-_-_-_-_-_-_-_-_-_-_-_-_-_-_-_-_-_-_-_-_-_-_-_-_-_-_-_-_-_-_-_-_-_-_-_-_-_-_-_-_-_-_-_-_-_-_-_-_-_-_-_-

%-_-_-_-_-_-_-_-_-_-_-_-_-_-_-_-_-_-_-_-_-_-_-_-_-_-_-_-_-_-_-_-_-_-_-_-_-_-_-_-_-_-_-_-_-_-_-_-_-_-_-_-_-_-_-_-_-_-_-_-_-_-_-_-_-_-_-_-_-_-_-_-_-_-_-_-
\section{Conclusions} \label{sec:conclus}

We have thus seen that in any category $\calC$, a concept of a 1-cell $F_1$ and boundary  maps from the terminal object into the 1-cell creates a notion of homotopy. Figure \ref{fig:hmtpy_outline} summarizes how under Axioms \ref{A:1_0_cell}, \ref{A:swap} and \ref{A:F1_join} on the 1-cell, the relation of homotopy is an equivalence relation which is invariant under composition. The key connection between homology and homotopy is the homotopy invariance of homology. This is enabled by Axiom \ref{A:convex}. Section \ref{sec:cnvxty} shows how convexity implies acyclicity, a key component for the connection discussed. 

In summary, a general category $\calC$ can under certain conditions be equipped with notions of homotopy and homology which play the analogous role as the case when $\calC= \Topo$. The conditions expressed in Axioms \ref{A:C} -- \ref{A:convex} capture certain aspects of $\Topo$. Among these, Axiom \ref{A:convex} on convexity is the hardest to realize. A homotopy invariant notion of homology provides a gateway for all theorems, techniques and numerical methods for Topology \cite[e.g.]{BubenikMilicevic2021, BubenikScott2014prstnt, BauerLesnick2020prsstnc, OtterEtAl2017roadmap} to be applied to a different category $\calC$. We end this article with a brief discussion of various related notions from Topology that may also be realized in the general category $\calC$.

\paragraph{Spheres and disks} We begin with a generalization of the notion of a \emph{sphere} and a \emph{disk}. In our general category $\calC$, the cell $F_{k+1}$ takes the place of the $k+1$-dimensional disk $D^{k+1}$, In $\Topo$ a $k$-dimensional sphere is the boundary of $(k+1)$-cell $D^{k+1}$. The sphere can be built by joining a collection of $k$-disks along their $(k-1)$-dimensional boundaries. This corresponds to a colimit of a diagram in $\Topo$, as shown below :
\[\begin{tikzcd}
	&&& \begin{array}{c} (k+1)-\mbox{disk} \\ D^{k+1} = F_{k+1} \end{array} \\
	\braces{\begin{array}{c} \mbox{Diagram involving} \\ k-1 \mbox{ and } k \mbox{ dimensional cells}
	\end{array} } \arrow[rrr, Rightarrow, "\mbox{limiting}", "\mbox{co-cone}"'] \arrow[urrr, bend left=15, pos=0.2, "\mbox{inclusions}", Rightarrow] &&& \begin{array}{c} k-\mbox{sphere} \\ S^{k} \end{array} \arrow[u, Rightarrow, "\subset_k"']
\end{tikzcd}\]
One can create an analogous diagram in $\calC$ with the same combinatorial content. The colimit of the diagram will be denoted as $S^k$ and interpreted as a $k$-dimensional sphere. The inclusion indicated by $\subset_k$ arises naturally as a consequence of the universality of colimits. 

\paragraph{CW-complexes} The notion of boundary of a cell enables the important notion of \textit{handle-attachment} or cell-gluing, Given a $\calC$-object $X$ and and integer $k\in \num$, a \emph{$k$-handle attachment} to $X$ is a colimit / pushout of the following diagram
\[\begin{tikzcd}
	&& F_k \arrow[drr] \\
	S^{k-1} \arrow[urr, "\subset_k"] \arrow[drr, "\alpha"'] && && X \sqcup_{\alpha} F_k \\
	&& X \arrow[urr]
\end{tikzcd}\]
In the diagram above, $\alpha$ is any morphism between the corresponding objects. The resulting colimit object $X \sqcup_{\alpha} F_k$ thus depends on the choice of $\alpha$. One can attach an arbitrary number of $k$-handles in this manner to $X$. The construction can be expressed by a similar colimit diagram
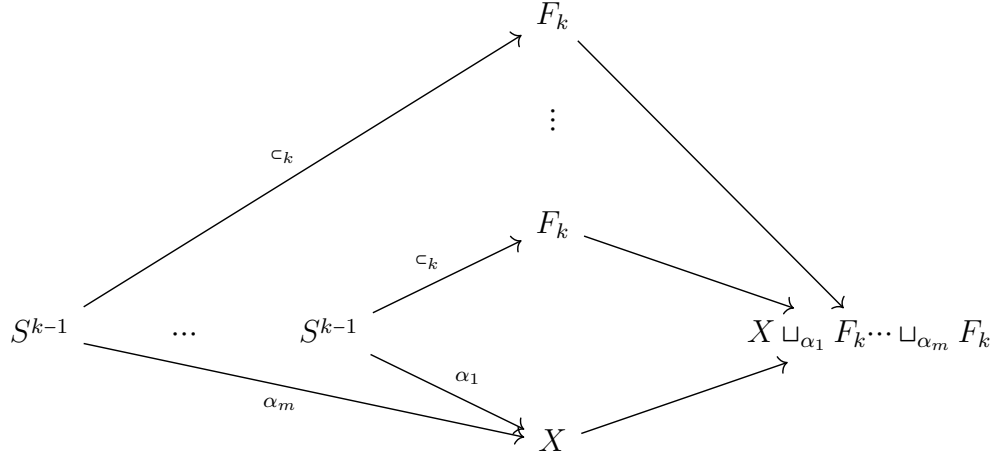
\[\begin{tikzcd}
	& & && F_k \arrow[dddrr] \\
	& & && \vdots \\
	& & && F_k \arrow[drr] \\
	S^{k-1} \arrow[uuurrrr, "\subset_k"] \arrow[drrrr, "\alpha_m"'] & \cdots & S^{k-1} \arrow[urr, "\subset_k"] \arrow[drr, "\alpha_1"] && && X \sqcup_{\alpha_1} F_k \cdots \sqcup_{\alpha_m} F_k \\
	& & && X \arrow[urr]
\end{tikzcd}\]
A \textit{CW complex} or \emph{cellular complex} \cite{Whitehead1949combin} in $\calC$ is constructed by taking the union of a sequence of $\calC$-objects
\[ X_{0} \xrightarrow{\subset} X_{1} \xrightarrow{\subset} \cdots X_{k-1} \xrightarrow{\subset} X_{k} \xrightarrow{\subset} \cdots \]
in which $X_{k-1}$ is a sub-object of $X_{k}$. The object $X_k$ is created recursively out of $X_{k-1}$ by taking a collection of $k$-cells $\SetDef{ F^{(i)}_k }{i\in I}$ and gluing each of them to $X_{k-1}$. These maps are called \emph{attachment maps}. The associated cellular complex $X$ is the colimit of the above sequence. If the sequence is finite and ends at some $X_N$, $N$ is called the dimension of the cellular object. The $k$-th object $X_k$ is called the $k$-th skeleton.

\paragraph{Cell complexes} A finite category $J$ will be called a \emph{cell-complex diagram} if it satisfies the following properties : 
\begin{enumerate} [(i)]
	\item The objects of $J$ can be partitioned into levels $J^{(1)}, \ldots, J^{(L)}$.
	\item The morphisms in $J$ are generated by arrows that connect an $l$-level object to an $(l+1)$-level object.
\end{enumerate}
A \emph{ cellular complex } in $\calC$ is a functor from $\Phi : J \to \calC$ such that
\begin{enumerate} [(i)]
	\item $\Phi$ maps each $l$-level object into the object $F(l)$.
	\item Each morphism $f$ in $J$ that connects an $l$-level object to an $(l+1)$-level object, is mapped into a face map $F \paran{ d_{l+1,j} }$ for some $0\leq j \leq l+1$.
\end{enumerate}
Thus a cellular complex is a specific type of a finite diagram in $\calC$. It respects the cells and face maps prescribed by the functor $F$ from \eqref{eqn:functorF}. By a reuse of terminology, we say that an object $X \in \calC$ is a cellular complex if there is a cellular complex $\Phi : J \to \calC$ such that $X = \colim \Phi$. In this case, the functor $\Phi$ will be called a \emph{cellular decomposition} of the object $X$ and $L$ its \emph{dimension}. 

Thus the notion of cells and boundary maps lead to the notion of spheres and disks. These in turn make possible the notions of CW-complexes and cellular complexes.

\paragraph{Reconstruction} There are certain situations in which the first few homology groups of $\calC$-objects may be calculable, from data or other inputs. Let $\SqBrack{ L-\text{complex} }$ denote the subcategory of $\calC$ composed of $L$-dimensional complexes of any of the two types described above. Now consider the following diagram :
\begin{equation} \label{eqn:TDA:scheme}
	\begin{tikzcd} [row sep = large]
		\calC \arrow[d, "\text{Homology}"'] && \SqBrack{ L-\text{simplex} } \arrow[dashed, Holud]{d}[name=n1]{} \arrow[ll, "\colim"'] \\
		\AbelCat^{\num} \arrow[rr, "\proj_L"'] && \AbelCat^{L} \arrow[bend right=60, Shobuj]{u}[swap, name=n2]{ \text{Reconstruct} }
		\arrow[shorten <=0.1pt, shorten >=4pt, -Bar, to path={(n1) to[out=0,in=180] node[xshift=10pt]{} (n2)} ]{  }
	\end{tikzcd}
\end{equation}
The left loop shows the interpretation of $L$ -simplices as $\calC$-objects, the calculation of their Homology, and finally restricting them to the their first $L$ indices. These signature of any $\calC$ object may be directly computable from data. The task of a data-driven reconstruction may be interpreted as the creation of the functor $\text{Reconstruct}$ shown in green, which must be the right adjoint to the composite functor shown in dashed yellow. This diagram presents how a Homology theory enables a reconstruction scheme for objects in terms of complexes, with the homology groups being the objectives.

Yet another proposition is to prove a categorical version of the Hurewicz homomorphism $h_n: \pi_n(X) \to H_n(X)$ and the \textit{Hurewicz Theorem}. This requires an expansion to higher homotopy groups using the higher order cells. A related task is a categorical version of the \textit{Whitehead Theorem} which establishes algebraic conditions for morphisms between CW complexes to be homotopy equivalences. These two theorems would provide a complete cellular theory for categories of more generality.

%-_-_-_-_-_-_-_-_-_-_-_-_-_-_-_-_-_-_-_-_-_-_-_-_-_-_-_-_-_-_-_-_-_-_-_-_-_-_-_-_-_-_-_-_-_-_-_-_-_-_-_-_-_-_-_-_-_-_-_-_-_-_-_-_-_-_-_-_-_-_-_-_-_-_-_-
\section{Appendix} \label{sec:proofs}

%................................................................................................................
\subsection{The Chain-functor} \label{sec:chain_map}

This section provides the full details of the nature of the $\Chain$-functor described in Diagram \eqref{eqn:CatHmlgy}. It maps each $\calC$-object $A$ to the chain complex in which the $n$-th object $\Chain_n(A)$ is the free Abelian group generated by $\Hom_{\calC} \paran{ F(n); A }$. Since $\Chain_n(A)$ is a free Abelian group, the boundary map $\partial^A_n : \Chain_n(A) \to \Chain_{n-1}(A)$ is completely determined by its action on the generators. This action is an algebraic sum defined in terms of the face maps \eqref{eqn:def:face} :
\begin{equation} \label{eqn:chain_bndry}
	\forall \begin{tikzcd} F(n) \arrow[d, "\sigma"'] \\ A \end{tikzcd} \,:\, \partial^A_n(\sigma) := \sum_{j=0}^n (-1)^j (\sigma \circ F(d_{n,j})) 
	,\quad
	\begin{tikzcd}
		F(n-1) \arrow[drr, dashed, Shobuj, "\sigma \circ F(d_{n,j})"'] \arrow[rr, "F(d_{n,j})"] && F(n) \arrow[d, "\sigma"] \\
		&& A
	\end{tikzcd}
\end{equation}
Each such generator $\sigma$ is called a \emph{singular $n$-simplex in} $A$. The diagram on the right shows how each such $\sigma$ is pre-composed with a face map to obtain a singular $(n-1)$-simplex. The identity $\partial^A_n \circ \partial^A_{n-1} = 0$ is an immediate consequence of the simplicial identities \eqref{eqn:smplc_id}.

We next verify the action of $\Chain$ in morphisms. So let $f: A \to B$ be a $\calC$-morphism. The induced chain map, typically denoted $f_\#: \Chain_\#(A) \to \Chain_\#(B)$, is constructed by composing $f$ with the basis elements of the singular chain groups :
\begin{equation} \label{eqn:chain_mrphsm}
	\forall n\in\num_0 \,, \forall \begin{tikzcd} F(n) \arrow[d, "\sigma"'] \\ A \end{tikzcd} 
	\begin{tikzcd} {} \arrow[rrr, mapsto, "\Chain(f)_n"] &&& {} \end{tikzcd}
	\forall n\in\num_0 \,, \forall \begin{tikzcd} F(n) \arrow[d, "\sigma"'] \arrow[dr, dashed, Shobuj, "\Chain(f)_n(\sigma)"] \\ A \arrow[r, "f"] & B \end{tikzcd}
\end{equation}
Thus generators are mapped into generators, and we get a well-defined Abelian group homomorphism
\[\Chain(f)_n : \Chain(A)_n \to \Chain(B)_n .\]
The definition \eqref{eqn:chain_mrphsm} also reveals the obvious compositionality of this transformation. It remains to check commutativity with the boundary operator, i.e., the commutations
\[\forall n\in\num \,:\, 
\begin{tikzcd}
	\Chain_n(A) \arrow[d, "\partial_n^A"'] \arrow[rr, "\Chain(f)_{n}"] && \Chain_{n}(B) \arrow[d, "\partial_n^B"] \\
	\Chain_{n-1}(A) \arrow[rr, "\Chain(f)_{n-1}"'] && \Chain_{n-1}(B)
\end{tikzcd}\]
Again, it is enough to verify this commutation of Abelian homomorphisms on the generators. For each generator $\sigma \in \Chain_n(A)$ we have
\[\begin{split}
	\Chain(f)_{n-1} \circ \partial_n^A(\sigma) &= \Chain(f)_{n-1} \paran{ \sum_{j=0}^n (-1)^j \sigma \circ F(d_{n,j}) } , \quad \mbox{by \eqref{eqn:chain_bndry}} \\
	&= \sum_{j=0}^n (-1)^j \Chain(f)_{n-1} \paran{ \sigma \circ F(d_{n,j}) }  = \sum_{j=0}^n (-1)^j f\circ \sigma \circ F(d_{n,j}) , \quad \mbox{by \eqref{eqn:chain_mrphsm}} \\
	&= \sum_{j=0}^n (-1)^j (f\circ \sigma) \circ F(d_{n,j}) = \partial_n^A \paran{ f\circ \sigma } , \quad \mbox{by \eqref{eqn:chain_bndry}} \\
	&= \partial_n^A \circ \Chain(f)_{n} (\sigma), \quad \mbox{by \eqref{eqn:chain_mrphsm}}.
\end{split}\]
%

%................................................................................................................
\subsection{Constructions from categorical products}

In this short section we briefly review some commutations created by categorical products. The following commutation can be taken as the definition of categorical products of morphisms :
\begin{equation} \label{eqn:od3l9}
	\begin{tikzcd} a \arrow[d, "f"'] \\ b \end{tikzcd} , \; 
	\begin{tikzcd} a' \arrow[d, "f'"'] \\ b' \end{tikzcd} 
	\imply 
	\begin{tikzcd}
		& a \arrow[r, "f"] & b \\
		a\times a' \arrow[rrr, "f\times f'", dotted] \arrow[ur, "\proj_2", Akashi] \arrow[dr, "\proj_1"', Akashi]  &&& b\times b' \arrow[dl, "\proj_2", Holud] \arrow[ul, "\proj_1"', Holud] \\
		& a' \arrow[r, "f'"'] & b'
	\end{tikzcd}
\end{equation}
The blue and yellow sub-diagrams are limiting cones, and the central dotted morphism is the unique morphism by which they transform. Equation \eqref{eqn:od3l9} shows how products extend from objects to morphisms. In fact, this is functorial : 
\begin{lemma} [Categorical product as a functor] \label{lem:product_id}
	In any category $\calS$ with finite products, for every $n\in \num$, the $n$-fold product is a functor from $\calS^n \to \calS$. Moreover, if $\calS$ has a terminal element $1_{\calS}$, then $1_{\calS} : \calS \to \calS$ is the identity functor.
\end{lemma}
The next commutation relation that we need involves commutation with the terminal element $1_{\calC}$ of $\calC$ : 
\begin{equation} \label{eqn:iojd9}
	\begin{tikzcd}
		a \arrow[dr, "="'] \arrow[r, "\cong"] & a \times 1_{\calC} \arrow[d, "\proj_1"] \\
		& a
	\end{tikzcd}
\end{equation}
It is a well known fact \citep[e.g.][Sec 4]{Riehl2017context} that products with the terminal object is an identity operation. The second diagonal equality arrow follows from the universality of the limit. We use these two basic commutations to derive some useful results required in our proofs.

The next identity is an application of identities \eqref{eqn:od3l9} and \eqref{eqn:iojd9} above. Given any morphism $\sigma$ as shown on the left below :
\[\begin{tikzcd} A \arrow[d, "\sigma"'] \\ X \end{tikzcd}
\,\Rightarrow\,
\begin{tikzcd}
	A \arrow[drrrr, bend right=10, "="', Itranga] \arrow[r, "=", Itranga] & A \times F(0) \arrow[rr, "A\times F\paran{ d_{1,j} }", Itranga] && A \times F(1) \arrow[dr, "\proj_2"']  \arrow[Itranga]{rrr}{ \sigma \times F(1) } &&& X \times F(1) \arrow[dl, "\proj_1"] \\
	&& & & A \arrow[r, "\sigma"] \arrow[dr, "\sigma"', Itranga] & X \\
	&& & & & X \arrow[u, "="] \arrow[r, "="', Itranga] & X \times F(0) \arrow[ul, "\proj_1"']  \arrow[bend right=20, Itranga]{uu}[swap]{X \times F\paran{ d_{1,j} }}
\end{tikzcd}\]
we have the commutation diagram on the right. Each commutation loop is a consequence of \eqref{eqn:od3l9} and \eqref{eqn:iojd9}. The two commuting paths traced out in red gives us an identity involving the homotopy face maps (Property \ref{P:4}) of $A$ and $X$ :
\begin{equation}
	\forall \begin{tikzcd} A \arrow[d, "\sigma"'] \\ X \end{tikzcd} 
	\,:\, 
	\lambda^X_j \circ \sigma = \paran{ \sigma \times F(1) } \circ \lambda^A_j , \quad j=0,1 .
\end{equation}
Substituting $A=F(0)=1$ and $A=F(n)$ respectively gives us two important identities :
\begin{equation} \label{eqn:kubsd9}
	\forall \begin{tikzcd} 1 \arrow[d, "\sigma"'] \\ X \end{tikzcd} 
	\,:\, 
	\lambda^X_j \circ \sigma =  \paran{ \sigma \times F(1) } \circ F\paran{ d_{1,j} } , \quad j=0,1 , 
\end{equation}
and
\begin{equation} \label{eqn:sl5g}
	\forall n\in \num_0 \,:\, \forall \begin{tikzcd} F(n) \arrow[d, "\sigma"'] \\ X \end{tikzcd} 
	\,:\, 
	\paran{ \sigma \times F(1) } \circ \lambda^{F(n)}_i = \lambda^X_j \circ \sigma : F(n) \to X\times F(1), \; j=0,1.
\end{equation}
These will be key to proving Lemma \ref{lem:P2_to_P5} later. The next useful result is based on the functorial property of $\Chain$ : 
\begin{equation} \label{eqn:chain:1}
	\begin{tikzcd} A \arrow[d, "f"'] \\ B \end{tikzcd} \imply 
	\begin{tikzcd}
		\Chain (A)_n \arrow{d}[swap]{ \Chain(f)_n } \arrow[r, "\partial^{A}_n"] & \Chain (A)_{n-1} \arrow{d}{ \Chain(f)_{n-1} } \\ 
		\Chain (B)_n \arrow[r, "\partial^{B}_n"'] & \Chain (B)_{n-1}
	\end{tikzcd}
\end{equation}
The commutation on the right is an outcome of an arbitrary morphism $f$ in $\calC$ as shown on the right. 

%................................................................................................................
\subsection{Proof of Lemma \ref{lem:P2_to_P5}} \label{sec:proof:P2_to_P5}

The proof follows the argument in \cite[Thm 4.23]{Rotman2013algtopo}. The proof will be by induction : 

\paragraph{Inductive hypothesis} For every object $X\in \calC$, there are maps $P^{X}_{n} : \Chain(X)_n \to \Chain\paran{ X\times F(1) }_{n+1}$ as in Property \ref{P:5} which satisfy the identities \eqref{eqn:P:5a}. Moreover, the following commutations hold :
\begin{equation} \label{eqn:P:5b}
	\forall n\in \num_0 \,:\;
	\begin{tikzcd}
		\Chain\paran{ F(n) }_n \arrow[rr, "P^{F(n)}_n"] \arrow[d, "\sigma\circ"'] && \Chain\paran{ F(n) \times F(1) }_{n+1} \arrow{d}{ \paran{\sigma\times F(1)} \circ } \\
		\Chain(X)_n \arrow[rr, "P^{X}_n"'] && \Chain\paran{ X \times F(1) }_{n+1}
	\end{tikzcd} 
	, \quad \forall \sigma \in \Hom_{\calC} \paran{ F(n); X } .
\end{equation}
Note that the inductive hypothesis for proving Property \ref{P:5} is stronger than Property \ref{P:5} itself. %Our goal is to construct a sequence of maps
%\[\begin{tikzcd} \Hom_{\calC} \paran{ F(n); X } \arrow[d, Shobuj, "Q^X_0"'] \\ \Hom_{\calC} \paran{ F(1); X \times F(1) }\end{tikzcd}\]
%

\paragraph{Base case} We begin with the case $n=0$. Now define
\begin{equation} \label{eqn:def:PX0}
	\begin{tikzcd} \Hom_{\calC} \paran{ F(0); X } \arrow[d, Shobuj, "P^X_0"'] \\ \Hom_{\calC} \paran{ F(1); X \times F(1) }\end{tikzcd}
	\quad \sigma \mapsto 
	\begin{tikzcd}
		F(1) \arrow[dr, dashed, Shobuj] \arrow[r, "\cong"] & 1_{\calC} \times F(1) 
		\arrow[d, "\sigma \times F(1)"] \\
		& X \times F(1)
	\end{tikzcd}
\end{equation}
The boundary map on this transform behaves as :
\[\begin{split}
	\partial_1 \paran{ P^X_0 \sigma } &= \sum_{i=0}^{1} (-1)^i \paran{ P^X_0 \sigma } \circ F\paran{ d_{1,i} } = \paran{ P^X_0 \sigma } \circ F\paran{ d_{1,0} } - \paran{ P^X_0 \sigma } \circ F\paran{ d_{1,1} } \\
	& = \paran{ \sigma \times F(1) } \circ F\paran{ d_{1,0} } - \paran{ \sigma \times F(1) } \circ F\paran{ d_{1,1} } = \lambda^X_1 \circ \sigma - \lambda^X_0 \circ \sigma, \quad \mbox{by \eqref{eqn:kubsd9}} , \\
	&= \Chain\paran{ \lambda^X_1 }_{0} \sigma - \Chain\paran{ \lambda^X_0 }_{0} \sigma .
\end{split}\]
This equality is simplified version of \eqref{eqn:P:5a} for the case $n=0$. Next, the commutation \eqref{eqn:P:5b} becomes :
\[\begin{tikzcd}
	\Chain\paran{ F(0) }_0 \arrow[rr, "P^{F(0)}_0"] \arrow[d, "\sigma\circ"'] && \Chain\paran{ F(0) \times F(1) }_{n+1} \arrow{d}{ \paran{ \sigma\times F(1) } \circ } \\
	\Chain(X)_n \arrow[rr, "P^{X}_0"'] && \Chain\paran{ X \times F(1) }_{1}
\end{tikzcd} \]
Note that the collection $\Chain\paran{ F(0) }_0$ has only one member -- the identity morphism $\Id_{F(0)}$. The counter-clockwise path of the above figure is
\[P^{X}_0 \paran{ \sigma\circ \Id_{F(0)} } = P^{X}_0 \paran{ \sigma } = \sigma \times F(1) , \]
by \eqref{eqn:def:PX0}. The clockwise path is 
\[ \paran{ \sigma\times F(1) } \circ \paran{ P^{F(0)}_0 \paran{ \Id_{F(0)} }  } = \paran{ \sigma\times F(1) } \circ \paran{ \Id_{F(0)} \times F(1) } = \sigma\times F(1) .\]
Thus the two paths are the same and the verification of the case $n=0$ has been complete.

\paragraph{Inductive step} We now assume that there is an $N\geq 0$ such that \eqref{eqn:P:5a} and  \eqref{eqn:P:5b} are true for all $n\leq N$. Henceforth we shall $I$ to abbreviate $F(1)$. Consider the following arrangement of groups and morphisms :
\[\begin{tikzcd}
	\Chain(X\times I)_{n+1} \arrow[rr, "\partial^{X\times I}_{N+1}"] &  & \Chain(X\times I)_{n} \arrow[rr, "\partial^{X\times I}_{N}"] &  & \Chain(X\times I)_{n-1} \\
	\Chain(X)_{n+1} \arrow[rr, "\partial^X_{N+1}"'] &  & \Chain(X)_{n} \arrow[llu, "P^X_{N}", dotted] \arrow[rr, "\partial^X_{N}"'] &  & \Chain(X)_{n-1} \arrow[llu, "P^X_{N-1}", dotted]
\end{tikzcd}\]
This is not a commutative diagram and is meant to provide an overview. Our focus will be on the homomorphism
\[ Q_N := \Chain\paran{ \lambda^{F(N)}_1 }_{N} - \Chain\paran{ \lambda^{F(N)}_0 }_{N} - P^{F(N)}_{N-1} \partial^{F(N)}_{N} : \Chain \paran{ F(N) }_N \to \Chain \paran{ F(N) \times I }_{N+1} . \]
Now note that
\[\begin{split}
	& \partial^{F(N)\times I}_N Q_N = \partial^{F(N)\times I}_{N}  \Chain\paran{ \lambda^{F(N)}_1 }_{N} - \partial^{F(N)\times I}_{N} \Chain\paran{ \lambda^{F(N)}_0 } - \partial^{F(N)\times I}_{N} \circ P^{F(N)}_{N-1} \partial_{N} \\
	& \quad \quad = \Chain\paran{ \lambda^{F(N)}_1 }_{N-1} \partial^{F(N)}_{N} - \Chain\paran{ \lambda^{F(N)}_0 }_{N-1} \partial^{F(N)}_{N} - \partial^{F(N) \times I}_{N} \circ P^{F(N)}_{N-1} \partial^{F(N)}_{N}, \quad \mbox{by \eqref{eqn:chain:1} } .
\end{split}\]
We shall now use our inductive assumption and make the substitution
\[ \partial^{F(N) \times I}_{N} P^{F(N)}_{N-1} = \Chain\paran{ \lambda^{F(N)}_1 }_{N-1} - \Chain\paran{ \lambda^{F(N)}_0 }_{N-1} - P^{F(N)}_{N-1} \partial^{F(N)}_{N-1} .\]
As a result we get :
\[\begin{split}
	\partial^{F(N)\times I}_N Q_N &= \left[ \Chain\paran{ \lambda^{F(N)}_1 }_{N-1} - \Chain\paran{ \lambda^{F(N)}_0 }_{N-1} - \Chain\paran{ \lambda^{F(N)}_1 }_{N-1} + \Chain\paran{ \lambda^{F(N)}_0 }_{N-1} \right.\\
	& \left. + P^{F(N)}_{N-1} \partial^{F(N)}_{N-1}  \right] \circ \partial^{F(N)}_{N} \\
	&=  P^{F(N)}_{N-1} \partial^{F(N)}_{N-1} \circ \partial^{F(N)}_{N} = 0.
\end{split}\]
This indicates that the entire image of the homomorphism $Q_N$ are cycles. We shall consider the image under $Q_N$ of the trivial morphism $\Id_{F(N)}$. This image $Q_N \paran{ \Id_{F(N)} }$ is an $N$-cycle of $F(N) \times F(1)$. By Axiom \ref{A:convex} both $F(1)$ and $F(N)$ are convex. Therefore by Theorem \ref{thm:cnvxty}~(i), $F(N) \times F(1)$ is convex too. Then by Theorem \ref{thm:cnvxty}~(ii) $F(N) \times F(1)$ has zero homology groups $H_k$ for $k>0$. In other words each cycle is also a boundary. Thus there is $\beta_{N+1} \in \Chain \paran{ F(N) \times F(1) }_{N+1} $ such that 
\begin{equation} \label{eqn:pdm39}
	\partial_{N+1} \beta_{N+1} = Q_N \paran{ \Id_{F(N)} }
\end{equation}
The map $P^{X}_{N+1}$ can now be defined as follows: 
\begin{equation} \label{eqn:def:PXN}
	\begin{tikzcd} \Chain(X)_{N+1} \arrow[d, "P^{X}_{N+1}"'] \\ \Chain\paran{ X\times I }_{N+2} \\  \end{tikzcd} \,:\,
	\begin{tikzcd} F(N+1) \arrow[d, Shobuj, "\sigma"'] \\ X \end{tikzcd} 
	\,\mapsto\,
	\begin{tikzcd} F(N+1) \times I \arrow[d, "\sigma\times I"'] \\ X  \times I\end{tikzcd} 
	\,\mapsto\,
	\begin{tikzcd} 
		F(N+1) \arrow[r, dotted, "\beta_{N+1}"] \arrow[dr, dashed, Shobuj, "P^{X}_{N+1}(\sigma)"'] & F(N+1) \times I \arrow[d, "\sigma\times I"] \\ 
		& X \times I
	\end{tikzcd} 
\end{equation}
The action of $P^{X}_{N+1}$ is explained on the right above. The commutation diagram shows a dotted arrow to indicate that $\beta_{N+1}$ is not quite a morphism but an algebraic sum of morphisms. The resulting composition, shown as a green dashed arrow, is also an algebraic sum of maps from $F(N+1)$ to $X\times I$ and thus an element of $\Chain\paran{ X\times I }_{N+2}$. We now check whether this $P^{X}_{N+1}$ so constructed satisfies \eqref{eqn:P:5a}. We have :
\[\begin{split}
	\partial_{N+1} P^{X}_{N+1} (\sigma) &= \partial_{N+1} \paran{ \sigma \times F(1) } \circ \beta_{N+1}, \quad \mbox{by \eqref{eqn:def:PXN}}, \\
	&= \partial_{N+1} \Chain \paran{ \sigma \times F(1) }_{N+2} \paran{ \beta_{N+1} } , \quad \mbox{by chain map transformation}, \\
	&= \Chain \paran{ \sigma \times F(1) }_{N+2} \partial_{N+1} \paran{ \beta_{N+1} } , \quad \mbox{by \eqref{eqn:chain:1} } ,\\
	&= \Chain \paran{ \sigma \times F(1) }_{N+2} Q_N \paran{ \Id_{F(N)} }, \quad \mbox{by \eqref{eqn:pdm39} } 
\end{split}\]
Next we utilize the definition of $Q_N$ to get 
\[\begin{split}
	\partial_{N+1} P^{X}_{N+1} (\sigma) &= \paran{ \sigma \times F(1) } \circ \SqBrack{ \Chain\paran{ \lambda^{F(N)}_1 }_{N} - \Chain\paran{ \lambda^{F(N)}_0 }_{N} - P^{F(N)}_{N-1} \partial_{N} } \paran{ \Id_{F(N)} } , \\
	&= \paran{ \sigma \times F(1) } \circ \lambda^{F(N)}_1 - \paran{ \sigma \times F(1) } \circ \lambda^{F(N)}_0 - \paran{ \sigma \times F(1) } \circ  P^{F(N)}_{N-1} \partial_{N} \paran{ \Id_{F(N)} } ,\\
	&= \paran{ \sigma \times F(1) } \circ \lambda^{F(N)}_1 - \paran{ \sigma \times F(1) } \circ \lambda^{F(N)}_0 - P^{F(N)}_{N-1} \circ \sigma \circ  \partial_{N} \paran{ \Id_{F(N)} }, 
\end{split}\]
where the last equality follows from the inductive assumption \eqref{eqn:P:5b} for $n=N$. We now proceed by simplifying the first two terms on the LHS using \eqref{eqn:sl5g} :
\[\begin{split}
	\partial_{N+1} P^{X}_{N+1} (\sigma) &= \lambda^{X}_1 \circ \sigma  - \lambda^{X}_0 \circ \sigma - P^{F(N)}_{N-1} \circ \sigma \circ  \partial_{N} \paran{ \Id_{F(N)} } \\
	&= \lambda^{X}_1 \circ \sigma  - \lambda^{X}_0 \circ \sigma - P^{F(N)}_{N-1} \circ \partial_{N} \sigma \circ \paran{ \Id_{F(N)} }, \quad \mbox{by \eqref{eqn:chain:1} } \\
	&= \SqBrack{ \lambda^{X}_1 - \lambda^{X}_0 - P^{F(N)}_{N-1} } \paran{ \sigma } ,
\end{split}\]
and thus \eqref{eqn:P:5a} is verified for $n=N+1$. It remains to verify \eqref{eqn:P:5b} for $n=N+1$. That involves testing the commutation loop on an $N$-simplex $\tau : F(n) \to F(n)$. The CW path gives
\[ \paran{\sigma \times F(1) } \circ P^{F(N)}_N (\tau) = \paran{\sigma \times F(1) } \circ \paran{ \tau \times F(1) } \circ \beta_{N+1} = \paran{ \sigma \tau \times F(1) } \circ \beta_{N+1} .\]
The CCW path gives
\[ P^{X}_N \paran{ \sigma\circ \tau  } = \paran{ \sigma \tau \times F(1) } \circ \beta_{N+1} . \]
Thus the two paths are equal, and \eqref{eqn:P:5b} stands verified for $n=N+1$. This completes the inductive reasoning, and completes the proof of Lemma \ref{lem:P2_to_P5}. \qed 
	
	%................................................................................................................
	%-_-_-_-_-_-_-_-_-_-_-_-_-_-_-_-_-_-_-_-_-_-_-_-_-_-_-_-_-_-_-_-_-_-_-_-_-_-_-_-_-_-_-_-_-_-_-_-_-_-_-_-_-_-_-_-_-_-_-_-_-_-_-_-_-_-_-_-_-_-_-_-_-_-_-_-
	%\bibliographystyle{unsrt_inline_url} \bibliography{\Path References,ref}

\end{document}